\documentclass[final]{siamltex}

\usepackage{graphicx}

\usepackage{color}
\usepackage{bm}
\usepackage{amsfonts}
\usepackage[tight,footnotesize]{subfigure}

\def\norm#1{\|#1\|}
\newcommand{\tr}{^{\sf T}}
\newcommand{\m}[1]{{\bf{#1}}}
\newcommand{\g}[1]{\bm #1}
\newcommand{\C}[1]{{\cal {#1}}}
\newcommand{\prox}{{\rm prox}}
\renewcommand{\bar}{\overline}

\newtheorem{remark}{Remark}[section]
\newtheorem{assumption}{Assumption}[section]

\title{Convergence rates for an inexact ADMM applied to
separable convex optimization
\thanks{
January 12, 2020, revised June 1, 2020.
The authors gratefully acknowledge support by the National
Science Foundation under grants 1819002 and 1819161, and
by the Office of Naval Research under grants N00014-15-1-2048 and
N00014-18-1-2100.
}}
\author{
    William W. Hager\thanks{{\tt hager@ufl.edu},
        http://people.clas.ufl.edu/hager/,
        PO Box 118105,
        Department of Mathematics,
        University of Florida, Gainesville, FL 32611-8105.
        Phone (352) 294-2308. Fax (352) 392-8357.}
\and
    Hongchao Zhang\thanks{{\tt hozhang@math.lsu.edu},
        http://math.lsu.edu/$\sim$hozhang/,
        Department of Mathematics,
        Louisiana State University, Baton Rouge, LA 70803-4918.
        Phone (225) 578-1982. Fax (225) 578-4276.}
}
\begin{document}
\maketitle
\begin{abstract}
Convergence rates are established for an inexact accelerated
alternating direction method of multipliers (I-ADMM) for general 
separable convex optimization with a linear constraint. 
Both ergodic and non-ergodic iterates are analyzed.
Relative to the iteration number $k$, the convergence rate is $\C{O}(1/k)$
in a convex setting and $\C{O}(1/k^2)$ in a strongly convex setting.
When an error bound condition holds, the algorithm is 2-step linearly
convergent.
The I-ADMM is designed so that the accuracy of the inexact iteration
preserves the global convergence rates of the {\it exact} iteration,
leading to better numerical performance in the test problems.
\end{abstract}


\pagestyle{myheadings}
\thispagestyle{plain}
\markboth{W. W. HAGER AND H. ZHANG}
{INEXACT ADMM FOR SEPARABLE CONVEX OPTIMIZATION}
\section{Introduction}

We consider a convex, separable linearly constrained optimization problem
\begin{equation}\label{Prob}
\min \; \Phi (\m{x}) \; \mbox{ subject to } \m{Ax} = \m{b},
\end{equation}
where $\Phi : \mathbb{R}^n \rightarrow \mathbb{R}\cup\{\infty\}$ and
$\m{A}$ is $N$ by $n$.
By a separable convex problem,
we mean that the objective function is a sum of $m$ independent parts,
and the matrix is partitioned compatibly as in
\begin{equation}\label{ProbM}
\Phi(\m{x}) = \sum_{i=1}^m f_i(\m{x}_i) + h_i(\m{x}_i)
\quad \mbox{and} \quad \m{Ax} = \sum_{i =1}^m \m{A}_i \m{x}_i.
\end{equation}
Here $f_i$ is convex and Lipschitz continuously differentiable,
$h_i$ is a proper closed convex function (possibly nonsmooth), and
$\m{A}_i$ is $N$ by $n_i$ with $\sum_{i=1}^m n_i = n$.
There is no column independence assumption for the $\m{A}_i$.
Constraints of the form $\m{x}_i \in \C{X}_i$, where $\C{X}_i$ is a closed
convex set, can be incorporated in the optimization problem by
letting $h_i$ be the indicator function of $\C{X}_i$.
That is, $h_i(\m{x}_i) = \infty$ when $\m{x}_i \not\in \C{X}_i$.
The problem (\ref{Prob})--(\ref{ProbM}) has attracted extensive research 
due to its importance in areas such as
image processing, statistical learning, and compressed sensing.
See the recent survey \cite{Boyd10} and its references.

It is assumed that there exists a solution $\m{x}^*$ to
(\ref{Prob})--(\ref{ProbM}) and an associated Lagrange multiplier
$\g{\lambda}^* \in \mathbb{R}^N$ such that the following
first-order optimality conditions hold:
$\m{Ax}^* = \m{b}$ and
for $i =$ $1, 2, \ldots, m$ and for all $\m{u} \in \mathbb{R}^{n_i}$, we have
\begin{equation}\label{Wstar}
\langle \nabla f_i (\m{x}_i^*) + \m{A}_i \tr \g{\lambda}^*,
\m{u} - \m{x}_i^* \rangle + h_i(\m{u}) \ge h_i (\m{x}_i^*),
\end{equation}
where $\nabla$ denotes the gradient.

A popular strategy for solving (\ref{Prob})--(\ref{ProbM})
is the alternating direction method of multipliers (ADMM)
\cite{GM76, gl84}: For $i = 1, \ldots, m$,
\begin{equation}\label{adm}
\quad \quad \left\{
\begin{array}{lcl}
\m{x}_i^{k+1} &\in&
\arg \displaystyle{\min_{\m{x}_i \in \mathbb{R}^{n_i}}}
\; \C{L}_\rho (\m{x}_1^{k+1}, \ldots,\m{x}_{i-1}^{k+1},\m{x}_i,
\m{x}_{i+1}^k, \ldots, \m{x}_m^k , \g{\lambda}^k ), \\
\g{\lambda}^{k+1} &=&  \g{\lambda}^k + \rho
(\m{Ax}^{k+1} - \m{b}), 
\end{array}
\right.
\end{equation}
where $\rho$ is a penalty parameter and $\C{L}_\rho$ is the augmented Lagrangian
defined by
\begin{equation}\label{AL}
\C{L}_\rho (\m{x}, \g{\lambda}) = \Phi(\m{x}) +
\langle \g{\lambda}, \m{Ax} - \m{b} \rangle +
\frac{\rho}{2} \| \m{Ax} - \m{b} \|^2.
\end{equation}

Early ADMMs only consider problem (\ref{Prob})--(\ref{ProbM})
with $m=2$ corresponding to a $2$-block structure.
In this case, the global convergence and complexity
can be found in \cite{EB92, HeYuan12}.
When $m \ge 3$, the ADMM strategy (\ref{adm})
is not necessarily convergent \cite{chyy2016},
although its practical efficiency has been observed in 
many recent applications \cite{TaoYuan2011,wgy2010}.
Many recent papers, including
\cite{CaiHanYuan17,ChenLiLiuYe15,ChenShenYou13,
DavisYin15, GoldfarbMa2012, HanYuan12, HTXY2013, HeTaoXuYuan12,
LiSunToh14, LinMaZhang15},
develop modifications to ADMM to ensure convergence when $m \ge 3$.
The approach we have taken employs a back substitution step to
complement the ADMM forward substitution step.
This modification was first introduced in \cite{HTXY2013, HeTaoXuYuan12}.

Much of the CPU time in an ADMM iteration is associated with the solution
of the minimization subproblems.
If $m = 1$, then ADMM reduces to the augmented Lagrangian method,
for which the first relative error criteria
based on the residual in an iteration emanates from \cite{rf76},
while more recent work includes \cite{Eckstein13, SolSva2000}.
For $m = 2$ or larger, inexact approaches to the ADMM subproblems
have been based on an absolute summable error criterion as in
\cite{ChenSunToh2015, EB92, GolTre1979}, a combined
adaptive/absolute summable error criterion \cite{LiLiaoYuan2013},
a relative error criteria \cite{EcksteinYao17,EcksteinYao18},
proximal regularizations \cite{ChenTeboulle94,HLHY2002}, and
linearized subproblems and reduced multiplier update steps
\cite{HongLuo2017}.

The approach taken in our I-ADMM emanates from our earlier work
\cite{chy13, hyz16, HagerZhang19} on a Bregman Operator Splitting algorithm with
a variable stepsize (BOSVS) with application to image processing.
In the current paper, the penalty term in the accelerated gradient algorithm of
\cite{HagerZhang19} is linearized so as to make the solution of the
I-ADMM subproblem trivial; there is essentially no reduction in the size of the
multiplier update step.
The I-ADMM is designed so that the accuracy of the inexact solution of the
ADMM subproblems is high enough to preserve the global convergence
rates of the {\it exact} iteration.
The global convergence results for I-ADMM are similar to those presented in \cite{HagerZhang19}.
However, there is no convergence rate analysis in \cite{HagerZhang19}.
In this paper, we focus on the convergence rate of I-ADMM.
In particular, 
relative to the iteration number $k$, the convergence rate for I-ADMM
is $\C{O}(1/k)$ for ergodic iterates in the convex setting and $\C{O}(1/k^2)$
for both ergodic and nonergodic iterates in a strongly convex setting.
When an error bound condition holds, I-ADMM is 2-step linearly convergent.
These convergence rates are consistent with those obtained for
ADMM schemes that solve subproblems exactly including the $\C{O}(1/k)$
rates in \cite{HeYuan12, RMS13, ShefiTeboulle14} for ergodic iterates,
and the linear rates obtained
in \cite{HanSunZhang17} and \cite{YangHan16} for a 2-block ADMM,
and in \cite{HongLuo2017} for the multi-block case and a sufficiently
small stepsize in the multiplier update.
For a more extensive review of linear convergence results for ADMMs,
see \cite{YuanZengZhang2020}. But again, almost all the
sublinear or linear convergence rate analysis is based on either 
a single linearization step to solve the subproblem 
or the exact solution of the (proximal) subproblem.
An advantage of our inexact scheme, compared to the exact iteration,
is that the computing time to achieve a given error tolerance is reduced,
while maintaining global convergence and its rate. 


The paper is organized as follows.
Section~\ref{algorithm} gives an overview of the inexact ADMM (I-ADMM)
that will be analyzed.
Section~\ref{aBOSVS} reviews the global convergence results found in
a companion paper \cite{HagerZhang19b}.
These global convergence results are similar to those established for the
inexact ADMM of \cite{HagerZhang19}.
Section~\ref{Sublinear} establishes a $\C{O}(1/k)$ convergence rate of
for ergodic iterates, and under a strong convexity assumption,
an $\C{O}(1/k^2)$ rate for both ergodic and nonergodic iterates.
Section~\ref{linearRate} gives 2-step linear convergence results
when an error bound condition holds.
Finally, Section~\ref{numerical} shows the observed convergence in some
image recovery problems.

\subsection{Notation}
Throughout the paper, $c$ denotes a generic positive constant which is
independent of parameters such as the
iteration number $k$ or the index $i \in [1, m]$.
Let $\C{W}^*$ denote the set of solution/multiplier pairs
$(\m{x}^*,  \g{\lambda}^*)$ 
of (\ref{Prob})--(\ref{ProbM}) satisfying (\ref{Wstar}),
while $(\m{x}^*, \g{\lambda}^*) \in \C{W}^*$ is
a generic solution/multiplier pair.
$\C{L}$ (without the $\rho$ subscript) stands for $\C{L}_0$.
For $\m{x}$ and $\m{y} \in \mathbb{R}^n$,
$\langle \m{x}, \m{y} \rangle = \m{x} \tr \m{y}$ is the standard inner product,
where the superscript $\tr$ denotes transpose.
The Euclidean vector norm, denoted $\norm{\cdot}$, is defined by
$\norm{\m{x}} =\sqrt{\langle \m{x}, \m{x} \rangle}$ and 
$\norm{\m{x}}_\m{G} =\sqrt{\m{x} \tr \m{G} \m{x}}$
for a positive definite matrix $\m{G}$.
For any matrix $\m{A}$,
the matrix norm induced by the Euclidean vector norm is the largest
singular value of $\m{A}$.
For a symmetric matrix, the Euclidean norm is the largest absolute eigenvalue.
In addition, $\m{A} \succ \m{0}$ and $\m{A} \succeq \m{0}$ mean that
the matrix $\m{A}$ is positive definite and positive semidefinite, respectively.
For a differentiable function $f: \mathbb{R}^n \to \mathbb{R}$,
$\nabla f (\m{x})$ is the gradient of $f$ at $\m{x}$, a column vector.
More generally, $\partial f (\m{x})$ denotes the subdifferential at $\m{x}$.
A function $h: \mathbb{R}^n \mapsto \mathbb{R}$
is convex with modulus $\mu \ge 0$ if
\[
h ((1-\theta)\m{x} + \theta \m{y}) \le
(1-\theta) h (\m{x}) + \theta h(\m{y})
- \theta (1-\theta)(\mu/2) \|\m{x} - \m{y}\|^2
\]
for all $\m{u}$ and $\m{v} \in \mathbb{R}^{n}$ and $\theta \in [0, 1]$.
If $\mu > 0$, then $h$ is strongly convex.
The prox operator associated with $h$ is defined by
\[
\mbox{prox}_h(\m{y}) = \arg \min_{\m{x} \in \mathbb{R}^n}
\left( h(\m{x}) + \frac{1}{2} \|\m{x} - \m{y}\|^2 \right) .
\]
%

\section{Algorithm Structure}
\label{algorithm}
The structure of our I-ADMM algorithm is given in
Algorithm~\ref{ADMMcommon}.
The algorithm generates sequences
$\m{x}^k$, $\m{y}^k$, $\m{z}^k$, and $R^k$.
Both $\m{x}^k$ and $\m{z}^k$ are updated in Step~1,
$R^k$ is updated in Step~2, and $\m{y}^{k}$ is updated in Step~3.
The error is estimated in Step~2.
The matrix $\m{Q}$ in Step 3 is an $m$ by $m$ block diagonal matrix
 whose $i$-th diagonal block, denoted $\m{Q}_i$, is chosen to satisfy
the conditions:
\begin{equation} \label{Q-def}
 \m{Q}_i \succ \m{0} \quad \mbox{ and } \quad 
 \bar{\m{Q}}_i := \m{Q}_i - \m{A}_{i}\tr \m{A}_{i} \succeq \m{0}.
\end{equation}
For example, we could take $\m{Q}_i = \gamma_i \m{I}$ where
$\gamma_i \ge \|\m{A}_{i}\tr \m{A}_{i}\|$.
Condition (\ref{Q-def}) is required for showing global convergence 
of our I-ADMM. Recent studies show that for the 2-block case ($m=2$) and
an exact ADMM, the requirement that $\bar{\m{Q}}_i$ is positive semidefinite
can be relaxed
\cite{CWHLv19, HeMaYuan2020}.
The matrix $\m{M}$ in Step~3 is the $m$ by $m$
block lower triangular matrix defined by
\begin{equation}\label{m-def}
\m{M}_{ij} = \left\{ \begin{array}{cl}
\m{A}_{i}\tr \m{A}_{j} & \mbox{if }  j < i , \\
\m{Q}_{i} & \mbox{if } j = i , \\
\m{0}          & \mbox{if } j > i.
\end{array} \right.
\end{equation}
By (\ref{Q-def}), $M$ is nonsingular.
The solution $\m{y}^{k+1}$ of the block upper triangular system
$\m{M} \tr (\m{y}^{k+1} - \m{y}^k) = \alpha \m{Q}(\m{z}^{k} - \m{y}^k)$
can be obtained by back substitution.
\renewcommand\figurename{Alg.}
\begin{figure}[h]
{\tt
\begin{tabular}{l}
\hline
{\bf Parameters:}
$\rho, \, \delta_{\min}, \, \theta_i > 0$,
$\alpha \in (0, 1)$, $\sigma \in (0, 1)$ \\[.05in]
{\bf Starting guess:} $\m{x}^1$ and $\g{\lambda}^1$. \\[.05in]
{\bf Initialize:} $\m{y}^1 = \m{x}^1$, $k = 1$ and
$\Gamma_i^0 = 0$, $1 \le i \le m$, $\epsilon^0 = \infty$ \\[.05in]
\begin{tabular}{ll}
{\bf Step 1:} &  For $i=1, \dots, m$  \\[.05in]
& $\quad$ Generate $\m{x}_i^{k+1}$, $\m{z}_i^k$, 
and $r_i^k$ by Algorithm~\ref{3}. \\
& End \\[.05in]
{\bf Step 2:}  & If 
$\epsilon^k := \theta_1 \|\m{z}^{k} - \m{y}^k\| + \theta_2 
\| \m{Az}^{k} - \m{b}\| + \theta_3 \sqrt{R^k}$ is sufficiently \\[.05in]
& small, then terminate, where $R^k = \sum_{i =1}^m {r}_i^k$. \\[.1in]
{\bf Step 3:} & Find $\m{y}^{k+1}$ by solving 
$\m{Q}^{-1} \m{M} \tr (\m{y}^{k+1} - \m{y}^k) = \alpha (\m{z}^{k} - \m{y}^k)$
\\[.05in]
& $\g{\lambda}^{k+1} = \g{\lambda}^k + \alpha \rho ( \m{Az}^{k} - \m{b} )$,
where $\m{Q}$ and $\m{M}$ are defined \\[.05in]
& in (\ref{Q-def}) and (\ref{m-def}), respectively. \\[.1in]
{\bf Step 4:} & $k:=k+1$, and go to Step 1.  \\
\end{tabular}\\
\hline
\end{tabular}
}
\caption{\rm I-ADMM algorithm.}
\label{ADMMcommon} 
\end{figure}
\renewcommand\figurename{Fig.}

In Step~1 of Algorithm~\ref{ADMMcommon}, we approximate the minimizer
in the $\m{x}_i$ subproblem of the ADMM algorithm (\ref{adm}) using
the accelerated gradient method of Algorithm~\ref{3}, which is a modification 
of Algorithm~5.1 in \cite{HagerZhang19b}.
Compared with Algorithm~5.1 in  \cite{HagerZhang19b},
Algorithm~\ref{3} has
a slightly different stopping condition in Step~1b, and
a proximal term to generate $\m{u}_i^l$ in Step~1a, where
\begin{equation}\label{bik}
\m{b}_i^k = \m{b} - \sum_{j < i} \m{A}_j \m{z}_j^{k} -
\sum_{j > i} \m{A}_j \m{y}_j^k .
\end{equation}

The termination condition for Algorithm~\ref{3} appears in Step~1b.
In this step, $\psi$ is a nonnegative function for which $\psi (0) = 0$
and $\psi (s) > 0$ for $s > 0$ with $\psi$ continuous at $s = 0$.
For example, $\psi(t) = t$.
Two different ways are developed in \cite{HagerZhang19} for choosing
\renewcommand\figurename{Alg.}
\begin{figure}[h]
{\tt
\begin{tabular}{l}
\hline
{\bf Inner loop of Step 1, an accelerated gradient method: }\\[.05in]
{\bf Initialize:} $\m{a}_i^0 = \m{u}_i^0 = \m{x}_i^k$ and $\alpha^1 = 1$.
\\[.05in]
{\bf For } $l = 1, 2, \ldots $\\[.05in]
\begin{tabular}{ll}
1a. & Choose $\delta^l \ge \delta_{\min}$ and when $l > 1$, choose
$\alpha^l \in (0, 1)$ such that\\[.05in]
& $\quad f_i(\bar{\m{a}}_i^{l}) +$
$\langle \nabla f_i(\bar{\m{a}}_i^{l}), \m{a}_i^{l} -\bar{\m{a}}_i^l \rangle
+ \frac{(1-\sigma)\delta^l}{2\alpha^l} \|\m{a}_i^{l} - \bar{\m{a}}_i^l\|^2
\ge f_i(\m{a}_i^{l})$,\\[.05in]
& where $\m{a}_i^l = (1-\alpha^l)\m{a}_i^{l-1} + \alpha^l \m{u}_i^l$,
$\bar{\m{a}}_i^l =
(1-\alpha^l)\m{a}_i^{l-1} + \alpha^l \m{u}_i^{l-1}$, and \\[.05in]
& $\m{u}_i^l =
\arg \min \{ P(\m{u})
+  \frac{\rho}{2} \|\m{u} - \m{y}_i^k \|_{\bar{\m{Q}}_i}^2
+ h_i (\m{u}): \m{u} \in \mathbb{R}^{n_i} \}$
with  \\[.05in]
& \quad $P(\m{u}) = 
\langle \nabla f_i(\bar{\m{a}}_i^{l}), \m{u} \rangle
+ \frac{\delta^l}{2} \|\m{u} - \m{u}_i^{l-1}\|^2 +
\frac{\rho}{2}\|\m{A}_i\m{u} - \m{b}_i^k + \g{\lambda}^k/\rho\|^2$, \\[.05in]
& and $\m{b}_i^k$ defined in (\ref{bik}). \\[.05in]
1b. &
If $\gamma^l =$
$(1/\delta^1) \displaystyle{\prod_{j=2}^l} (1-\alpha^j)^{-1}$
$\ge \Gamma_i^{k-1} $, where $\gamma^1 = 1/\delta^1$,\\[.05in]
& and $\| \m{a}_i^l - \m{x}_i^k \|/\sqrt{\gamma^l} \le \psi (\epsilon^{k-1})$, 
then break.\\[.05in]
\end{tabular}\\
{\bf Next} \\[.05in]
1c. Set $\m{x}_i^{k+1}= \m{u}_i^l$,  $\m{z}_i^k = \m{a}_i^l$,
$\Gamma_i^k = \gamma^l$, and
$r_i^k = (1/\Gamma_i^k) \sum_{j=1}^l \|\m{u}_i^j - \m{u}_i^{j-1}\|^2$.\\[.05in]
\hline
\end{tabular}
}
\caption{Inner loop in Step~$1$ of Algorithm~$\ref{ADMMcommon}$.}
\label{3} 
\end{figure}
\renewcommand\figurename{Fig.}
the parameters $\delta^l$ and $\alpha^l$ in Step~1a.
If a Lipschitz constant $\zeta_i$ of $f_i$ is known, then we could take
\begin{equation}\label{AG_constant}
\delta^l = \frac{1}{(1-\sigma)}\frac{2 \zeta_i}{l } \quad \mbox{and}
\quad \alpha^l = \frac{2}{l+1} \in (0,1],
\end{equation}
in which case, we have
\[
\frac{(1-\sigma) \delta^l}{ \alpha^l} = \frac{(l+1) \zeta_i}{l} > \zeta_i.
\]
This relation along with a Taylor series expansion of $f_i$
around $\bar{\m{a}}_i^l$ implies that the
line search condition in Step~1a of Algorithm~\ref{3} is satisfied for each $l$.

A different, adaptive way to choose to choose $\delta^l$ and $\alpha^l$,
that does not require knowledge
of the Lipschitz constant for $f_i$, is the following:
Choose $\delta_0^l \in [\delta_{\min}, \delta_{\max}]$, where
$0 < \delta_{\min} < \delta_{\max} < \infty$ are fixed constants,
independent of $k$ and $l$, and set
\begin{eqnarray}
\delta^l &=&
\frac{2}{\theta^l + \sqrt{(\theta^l)^2 +
4 \theta^l \Lambda^{l-1}}} \quad \mbox{and}
\quad \alpha^l = \frac{1}{1 + \delta^l \Lambda^{l-1}}, \quad \mbox{where}
\label{AG_linesearch}\\
\Lambda^l &=& \sum_{i=1}^l 1/\delta^i, \quad \Lambda^0 = 0,
\quad \mbox{and}\quad
\theta^l = 1/(\delta_0^l \eta^j) \mbox{ with } \eta > 1.
\nonumber
\end{eqnarray}
Here the integer $j\ge0$ is chosen as small a possible while
satisfying the inequality in Step~1a.
It can be shown that
\begin{equation}\label{delta/alpha}
\frac{\delta^l}{\alpha^l} = \frac{1}{\theta^l} = \delta_0^l \eta^j.
\end{equation}
Since $\eta > 1$, the ratio $\delta^l/\alpha^l$ appearing in
Step~1a tends to infinity as $j$ tends to infinity;
consequently, the inequality in Step~1a is satisfied for $j$ sufficiently large.

The stopping condition in Step~1b is elucidated using the following function:
\begin{eqnarray}
\bar{L}_i^k(\m{u}) &=&
L_i^k(\m{u}) + \frac{\rho}{2}  (\m{u} - \m{y}_i^k) \tr
\bar{\m{Q}}_i (\m{u} - \m{y}_i^k), \quad \mbox{where} \label{barlik}\\
L_i^k(\m{u}) &=& f_i (\m{u}) +
h_i (\m{u}) +
\frac{\rho}{2} \|\m{A}_i \m{u} - \m{b}_i^k + \g{\lambda}^k/\rho\|^2, \nonumber
\end{eqnarray}
and $\m{b}_i^k$ is defined in (\ref{bik}).
As pointed out in Lemma~\ref{lem-AG-Conv} of the next section,
for either of the parameter choices
(\ref{AG_constant}) or (\ref{AG_linesearch}),
the iterates $\m{a}_i^l$ of Algorithm~\ref{3} converge
to the minimizer of the function $\bar{L}_i^k$ at rate
$\C{O}(1/l)$, while the objective values converge at rate $\C{O}(1/l^2)$,
which is optimal for first-order methods applied to general convex,
possibly nonsmooth optimization problems.
We let $l_i^k$ denote the terminating value of $l$ in Step~{\rm 1b}.
\begin{remark}\label{stop_condition}
For the two parameter choices $(\ref{AG_constant})$ and $(\ref{AG_linesearch})$,
it has been shown {\rm \cite[pp.~227--228]{HagerZhang19}} that in Step~{\rm 1b},
$\gamma^l \ge l^2 \Theta$ for some constant $\Theta > 0$, independent of $k$
and $l$.
Consequently, the conditions in Step~{\rm 1b} are satisfied for $l$ sufficiently
large.
\end{remark}

\section{Global Convergence}
\label{aBOSVS}
The global convergence analysis of the accelerated ADMM in this
paper with a linearized penalty term is similar to the global convergence
analysis of the accelerated scheme in \cite{HagerZhang19}.
Hence, this section simply states the main results,
while the Appendix provides the detailed analysis.
The first result concerns the convergence of the iterates in Step~1 of I-ADMM
under the assumption that the sequence
\[
\xi^l := \delta^l \alpha^l \gamma^l
\]
is nondecreasing.
For either of the parameter choices
(\ref{AG_constant}) or (\ref{AG_linesearch}), it is shown
in \cite[pp.~227--228]{HagerZhang19} that $\xi^l = 1$.

\begin{lemma}\label{lem-AG-Conv}
If the sequence $\xi^l$ is nonincreasing,
then for each $i \in [1,m]$ and $L \ge 1$, we have
\smallskip
\begin{equation}\label{AG-Converge}
\quad \quad
\rho \nu_i \| \m{a}_i^{L} - \bar{\m{x}}_i^k\|^2 
+ \frac{\mu_{h,i}}{2} \sum_{l=1}^L
\|\bar{\m{x}}_i^k-\m{a}_i^L\|^2
+ \frac{\sigma}{\gamma^{L}}
\sum_{l=1}^{L}  \xi^l \|\m{u}_i^l- \m{u}_i^{l-1}\|^2 
\le  \frac{\| \m{x}_i^k - \bar{\m{x}}_i^k\|^2}
{\gamma^{L}},
\end{equation}
where $\mu_{h,i}$ is the modulus of convexity of $h_i$,
$\nu_i>0$ is the smallest eigenvalue of $\m{Q}_i$, and
\begin{equation}\label{xbar}
\bar{\m{x}}_i^k= \arg \min \{ \bar{L}_i^k(\m{u}) : \m{u} \in \mathbb{R}^{n_i}\}.
\end{equation}
\end{lemma}
\smallskip

Since $\bar{L}_i^k$ is strongly convex, it has a unique minimizer.
The following decay property plays an important role in the
global convergence analysis.
\smallskip
\begin{lemma}\label{L-key-lemma3}
Let $(\m{x}^*, \g{\lambda}^*) \in \C{W}^*$ be any solution/multiplier
pair for $(\ref{Prob})$--$(\ref{ProbM})$, let
$\m{x}^k$, $\m{y}^k$, $\m{z}^k$, $\m{u}_{k}^l$, and $\g{\lambda}^k$
be the iterates generated by Algorithm~$\ref{ADMMcommon}$,
and define
\begin{eqnarray}
E_k &=& \rho \| \m{y}^k - \m{x}^*\|_{\m{P}}^2 +
\frac{1}{\rho} \| \g{\lambda}^k - \g{\lambda}^*\|^2 +  \alpha 
\sum_{i=1}^m  \frac{\|\m{x}_i^k - \m{x}_i^*\|^2}{\Gamma_i^k} \quad \mbox{and}
\label{def-Ek} \\
E_k^- &=& \rho \| \m{y}^k - \m{x}^*\|_{\m{P}}^2 +
\frac{1}{\rho} \| \g{\lambda}^k - \g{\lambda}^*\|^2 +  \alpha 
\sum_{i=1}^m  \frac{\|\m{x}_i^k - \m{x}_i^*\|^2}{\Gamma_i^{k-1}},
\nonumber
\end{eqnarray}
where $\m{P} = \m{MQ}^{-1}\m{M}\tr$.
If $\xi^l := \delta^l \alpha^l \gamma^l = 1$ for each $l$, then
\begin{eqnarray}
&E_k   - E_{k+1} \ge  E_k - E_{k+1}^- \ge \label{Ek-decay} \\
& \alpha \left( 2\Delta^k + \sigma R^k
+ \rho(1-\alpha) (\|\m{y}^k - \m{z}^{k}\|_{\m{Q}}^2
+ \|\m{Az}^{k} - \m{b}\|^2)
+ \sum_{i=1}^m \mu_{h,i} \|\m{z}_i^k - \m{x}_i^*\|^2 \right), &
\nonumber
\end{eqnarray}
where $R^k$ is the residual defined in Step~$2$,  $\mu_{h,i}$ is the modulus of convexity of $h_i$, 
and
\begin{equation} \label{Deltak-form}
\Delta^k = \C{L}(\m{z}^k, \g{\lambda}^*) - \Phi(\m{x}^*) \ge 0.
\end{equation}
\end{lemma}
\smallskip

Recall that $\C{L} = \C{L}_0$ is the ordinary Lagrangian associated with
(\ref{Prob}).
This decay property is used to obtain the following global convergence result
for I-ADMM.
\smallskip

\begin{theorem}\label{L-glob-thm3}
Suppose the parameters $\delta^l$ and $\alpha^l$
in Algorithm~$\ref{3}$ are chosen
according to either $(\ref{AG_constant})$ or $(\ref{AG_linesearch})$.
If I-ADMM performs an infinite number of iterations
generating $\m{y}^k$, $\m{z}^k$, and $\g{\lambda}^k$,
then the sequences
$\m{y}^k$ and $\m{z}^k$ both approach a common limit $\m{x}^*$,
$\g{\lambda}^k$ approaches a limit $\g{\lambda}^*$, and
$(\m{x}^*, \g{\lambda}^*) \in \C{W}^*$.
\end{theorem}
\smallskip

Theorem~\ref{L-glob-thm3} considers the case of an infinite number
of iterations.
The following lemma considers the case where $\epsilon^k = 0$ within a finite
number of iterations.
\smallskip

\begin{lemma}\label{L-stop-cond3}
If $\epsilon^{k}=0$ in Algorithm~$\ref{ADMMcommon}$,
then $\m{x}^{k+1} = \m{x}^k = \m{y}^k = \m{z}^k$ solves
$(\ref{Prob})$--$(\ref{ProbM})$ and $(\m{x}^k, \g{\lambda}^k) \in \C{W}^*$.
\end{lemma}
\smallskip

\begin{proof}
If $\epsilon^k = 0$, then $r_i^k = 0$ for each $i$.
It follows that
\begin{equation}\label{equals}
\m{x}_i^k = \m{u}_i^0 = \m{u}_i^1 = \ldots = \m{u}_i^l.
\end{equation}
By Step~1c, $\m{u}_i^l = \m{x}_i^{k+1}$.
By the definitions
$\m{a}_i^l = (1-\alpha^l)\m{a}_i^{l-1} + \alpha^l \m{u}_i^l$ and
$\bar{\m{a}}_i^l = (1-\alpha^l)\m{a}_i^{l-1} + \alpha^l \m{u}_i^{l-1}$
where $\m{a}_i^0 = \m{u}_i^0 = \m{x}_i^k$, we have
$\m{a}_i^l = \bar{\m{a}}_i^l = \m{x}_i^k$ for each $l$ due to (\ref{equals}).
Again, by Step~1c, $\m{z}_i^k = \m{x}_i^k$.
Consequently, we have $\m{x}^{k+1} = \m{x}^k = \m{z}^k$.

Let $\m{x}^*$ denote $\m{x}^k$. Then $\m{x}^* =\m{x}^{k+1} = \m{x}^k = \m{z}^k$.
Since  $\epsilon^k = 0$, Step 2 of Algorithm~\ref{ADMMcommon} implies that 
$\m{y}^k = \m{z}^k = \m{x}^*$ and $\m{Ax}^* = \m{b}$.
Consequently, we have
\[
\m{b}_i^k = \m{b} - \sum_{j < i} \m{A}_j \m{z}_j^{k} -
\sum_{j > i} \m{A}_j \m{y}_j^k =
\m{b} - \sum_{j < i} \m{A}_j \m{x}_j^{*} -
\sum_{j > i} \m{A}_j \m{x}_j^* = \m{A}_i \m{x}_i^* .
\]
With this substitution in $P(\m{u})$ in Step~1a,
it follows that $\m{u}_i^l = \m{x}_i^*$ minimizes over
$\m{u}$ the function
\[
\langle \nabla f_i(\m{x}_i^{*}), \m{u} \rangle
+ \frac{\delta^l}{2} \|\m{u} - \m{x}_i^*\|^2 +
\frac{\rho}{2}\|\m{A}_i(\m{u} - \m{x}_i^*) + \g{\lambda}^k/\rho\|^2
+ \frac{\rho}{2} \|\m{u} - \m{x}_i^* \|_{\bar{\m{Q}}_i}^2
+ h_i (\m{u}) .
\]
The first-order optimality condition for this minimizer $\m{x}_i^*$
is the same as the first-order optimality condition
(\ref{Wstar}), but with $\g{\lambda}^*$ replaced by $\g{\lambda}^k$.
Hence, $(\m{x}^*, \g{\lambda}^k) \in \C{W}^*$.
\end{proof}

\begin{remark}
In this paper, we have focused on algorithms based on an inexact minimization
of $\bar{L}_i^k$ in Step~$1$ of Algorithm~\ref{ADMMcommon}.
In cases where $f_i$ and $h_i$ are simple enough that the exact minimizer
$\bar{\m{x}}_i^k$ of $\bar{L}_i^k$ can be quickly evaluated,
we could simply set $\m{x}_i^{k+1} = \m{z}_i^k = \bar{\m{x}}_i^k$, and
$r_i^k = 0$ in Step~$1$ of I-ADMM, and proceed to Step~2.
The global convergence results still hold.
\end{remark}

\section{Sublinear Convergence Rates}
\label{Sublinear}
In this section, sublinear convergences rates are established for I-ADMM. 
We first establish an $\C{O}(1/t)$ convergence rate for the ergodic iterates
\begin{equation} \label{ergodic-iterate}
\bar{\m{z}}^t = \frac{1}{t}\sum_{k=1}^{t}\m{z}^{k}
\end{equation}
generated by I-ADMM. 

\begin{theorem}
\label{sub1}
Let $(\m{x}^*, \g{\lambda}^*) \in \C{W}^*$ be any primal/dual solution pair
for $(\ref{Prob})$--$(\ref{ProbM})$ and let $\m{z}^k $ be generated by I-ADMM
with $\delta^l\alpha^l\gamma^l = 1$ for each $l$ and $k$.
Then, we have
\[
\C{L}(\bar{\m{z}}^t, \g{\lambda}^*) - \Phi(\m{x}^*) \le \frac{E_1}{2 \alpha t},
\]
where $\bar{\m{z}}^t$ is defined in $(\ref{ergodic-iterate})$
and $E_k$ is defined in $(\ref{def-Ek})$.
\end{theorem}
\begin{proof}
Discarding several nonnegative terms from (\ref{Ek-decay}),
we have 
\[
2 \alpha \Delta^k + E_{k+1} \le E_k.
\]
Adding this inequality over $k$ between 1 and $t$ yields
\[
2 \alpha \sum_{k=1}^t \Delta^k + E_{t+1} \le E_1.
\]
Hence, by the definition of $\Delta^k$ in (\ref{Deltak-form}), we have
\[
2 \alpha \sum_{k=1}^t \left[
\C{L}(\m{z}^k, \g{\lambda}^*) - \Phi(\m{x}^*) \right] \le E_1.
\]
By the convexity of $\Phi$
and the definition (\ref {ergodic-iterate}), it follows that
\[
2 \alpha t \left[
\C{L}( \bar{\m{z}}^t, \g{\lambda}^*) - \Phi(\m{x}^*) \right] \le E_1.
\]
This completes the proof.
\end{proof}

Note that the minimum of $\C{L}(\m{x}, \g{\lambda}^*)$
over $\m{x} \in \mathbb{R}^n$ is attained at $\m{x} = \m{x}^*$, and
$\C{L}(\m{x}^*, \g{\lambda}^*) = \Phi(\m{x}^*)$.
Hence, Theorem~\ref{sub1} bounds the difference between
$\C{L}(\bar{\m{z}}^t, \g{\lambda}^*)$ and
the minimum of $\C{L}(\cdot, \g{\lambda}^*)$.
We will strengthen the convergence rate to $\C{O}(1/t^2)$ when a strong
convexity assumption holds, and also obtain a convergence rate for
nonergodic iterates.
\begin{assumption}\label{strongconvex}
If $\mu_{f,i} \ge 0$ and $\mu_{h,i} \ge 0$ are the convexity moduli of 
$f_i$ and $h_i$ respectively, then
\begin{equation}\label{mufh}
\mu = \min \; \{ \mu_{f,i} + 3 \mu_{h,i} : i=1, \ldots, m \} >0.
\end{equation}
\end{assumption}

In the following theorem, we suppose that at the $k$-th iteration,
the penalty parameter $\rho$ is chosen in the following way:
\begin{equation}\label{adp-parameter}
\rho_k = (k_0 + k) \theta,
\end{equation}
where 
\begin{equation}\label{def-theta}
 \theta = \frac{ \alpha \mu}{8 \|\m{P}\|} \quad \mbox{and} \quad 
k_0 = \frac{4 \|\m{Q}^{-1/2} \m{P} \m{Q}^{-1/2}\|}{\alpha (1-\alpha)},
\end{equation}
with $\mu$ defined in Assumption \ref{strongconvex}, $\alpha \in (0,1)$ 
is the parameter in Algorithm~\ref{ADMMcommon}, and
$\m{P} = \m{MQ}^{-1}\m{M}\tr$.
We have the following theorem:
\begin{theorem}
Let $(\m{x}^*, \g{\lambda}^*) \in \C{W}^*$ be any solution/multiplier
pair for $(\ref{Prob})$--$(\ref{ProbM})$,
let $ \m{x}^k, \m{y}^k, \m{z}^k$ and $\g{\lambda}^k $
be generated by I-ADMM, and assume that
Assumption~$\ref{strongconvex}$ holds and
$\delta^l\alpha^l\gamma^l = 1$ for each $l$ and $k$.
Suppose that for every $k$, $\rho_k$ is given by $(\ref{adp-parameter})$
and $\Gamma_i^k$ satisfies 
\begin{equation} \label{Strong-Gamma}
\frac{k}{\Gamma_i^k} \ge \frac{k+1}{\Gamma_i^{k+1}},
\quad 1 \le i \le m. 
\end{equation}
Then, for all $t > 0$, we have
\begin{equation} \label{acc-rate}
\C{L}( \tilde{\m{z}}^t, \g{\lambda}^*) - \Phi(\m{x}^*) \le
\frac{2 \bar{c}}{ \alpha [t (t+1) + 2k_0 t]} 
\end{equation}
and 
\begin{equation} \label{y2-conv}
\|\m{y}^{t+1} - \m{x}^*\|^2 \le \frac{\bar{c}}{(t+k_0)^2 \theta },
\end{equation}
where
\begin{equation} \label{def-tilde-z}
\tilde {\m{z}}^t = \frac{2}{t(t+1)+ 2k_0t} \sum_{k=1}^t ((k_0+k) \m{z}^k),
\end{equation}
and 
\begin{equation}\label{cbar}
\bar{c} = \frac{1}{\theta} \| \g{\lambda}^1 - \g{\lambda}^*\|^2
+ \alpha (k_0 +1) \sum_{i=1}^m   \frac{\|\m{x}_i^1 - \m{x}_i^*\|^2}{\Gamma_i^1}
+k_0^2 \theta  \| \m{y}^1 - \m{x}^*\|_{\m{P}}^2 .
\end{equation}
\end{theorem}
\begin{proof}
By Assumption~\ref{strongconvex} and the definition (\ref{Deltak-form})
of $\Delta^k$, we have
\[
\Delta^k = \C{L}(\m{z}^k, \g{\lambda}^*) - \C{L}(\m{x}^*, \g{\lambda}^*)
\ge \sum_{i=1}^m \frac{\mu_{f,i} + \mu_{h,i}}{2} \|\m{z}_i^k - \m{x}_i^*\|^2
= \sum_{i=1}^m \frac{\mu_{f,i} + \mu_{h,i}}{2} \|\m{z}_{e,i}^k \|^2,
\]
where $\m{z}_e^k = \m{z}^k - \m{x}^*$.
The inequality (\ref{Ek-decay}) of Lemma~\ref{L-key-lemma3}
relates the error in two consecutive iterations,
where the $\rho$ in (\ref{Ek-decay}) is the penalty at iteration $k$.
Combining this with the definition of $\mu$ in
Assumption~\ref{strongconvex}, we have
\begin{eqnarray}\label{eqrm}
&& \alpha \left(
\Delta^k + \frac{\mu}{2} \|\m{z}_e^k\|^2 +
\rho_k (1-\alpha) \|\m{y}^k - \m{z}^{k}\|_{\m{Q}}^2 \right) \\
& \le & 
\rho_k (\| \m{y}_e^k \|_{\m{P}}^2
- \| \m{y}_e^{k+1} \|_{\m{P}}^2) +
\frac{1}{\rho_k} (\| \g{\lambda}_e^k \|^2
- \| \g{\lambda}_e^{k+1} \|^2)
+  \alpha \sum_{i=1}^m 
\frac{\|\m{x}_{e,i}^k \|^2 - \|\m{x}_{e,i}^{k+1} \|^2}
{\Gamma_i^k}, \nonumber
\end{eqnarray}
where $\m{x}_e^k = \m{x}^k - \m{x}^*$,
$\m{y}_e^k = \m{y}^k - \m{x}^*$, and
$\g{\lambda}_e^k = \g{\lambda}^k - \g{\lambda}^*$.

For any matrix $\m{P}$, it follows from an eigendecomposition that
\[
\m{x}\tr \m{x} \ge \frac{\m{x}\tr\m{Px}}{\|\m{P}\|} \quad \mbox{and} \quad
\m{x}\tr\m{Q} \m{x} \ge \frac{\m{x}\tr\m{Px}}
{\|\m{Q}^{-1/2}\m{P}\m{Q}^{-1/2}\|}.
\]
The second inequality is deduced from the first when
$\m{x}$ is replaced by $\m{Q}^{1/2}\m{x}$ and
$\m{P}$ is replaced by $\m{Q}^{-1/2}\m{P}\m{Q}^{-1/2}$.
This yields the following lower bound for terms on the left side of
(\ref{eqrm}):
\begin{eqnarray}
\frac{\mu}{2} \|\m{z}_e^k \|^2 +
\rho_k (1-\alpha) \|\m{y}^k - \m{z}^{k}\|_{\m{Q}}^2 &\ge&
\frac{\mu}{2\|\m{P}\|} \|\m{z}_e^k\|_{\m{P}}^2 +
\frac{\rho_k (1-\alpha)}{\|\m{Q}^{-1/2}\m{P}\m{Q}^{-1/2}\|}
\|\m{y}^k - \m{z}^{k}\|_{\m{P}}^2 \nonumber \\
&\ge&
\frac{\mu}{2\|\m{P}\|} \left( \|\m{z}_e^k \|_{\m{P}}^2 +
\|\m{y}^k - \m{z}^{k}\|_{\m{P}}^2 \right) \nonumber \\
&\ge&
\frac{\mu}{2\|\m{P}\|} \left(
2\|\m{z}_e^k \|_{\m{P}}^2 + \|\m{y}_e^k\|_{\m{P}}
- 2\|\m{z}_e^k\| \|\m{y}_e^k\| \right) \nonumber \\
&\ge& 
\frac{\mu}{4\|\m{P}\|} \|\m{y}_e^k\|_{\m{P}} =
\frac{2\theta}{\alpha} \|\m{y}_e^k\|_{\m{P}}. \label{h88}
\end{eqnarray}
The second inequality is due to the special form of $\rho_k$ in
(\ref{adp-parameter}) and (\ref{def-theta}),
and the last inequality is due to the relation
\[
ab \le \frac{1}{2} \left (2a^2 + \frac{1}{2} b^2 \right) .
\]

The inequality (\ref{h88}) is incorporated in the left side of (\ref{eqrm}).
We multiply the resulting inequality by $K := k_0 + k$, substitute
$\rho_k = K\theta$, exploit the assumption (\ref{Strong-Gamma}) and
the inequality $K(K - 2) \le (K - 1)^2$ to obtain
\begin{eqnarray*}
\alpha K \Delta^k 
& \le & \theta \left( (K-1)^2 \| \m{y}_e^k \|_{\m{P}}^2 - K^2
\| \m{y}_e^{k+1}\|_{\m{P}}^2 \right) +
\frac{1}{\theta} (\| \g{\lambda}_e^k \|^2 - \|
\g{\lambda}_e^{k+1} \|^2) \\
&&  + \alpha \sum_{i=1}^m \left(
\frac{K\|\m{x}_{e,i}^k \|^2}{\Gamma_i^k}
- \frac{(K+1)\|\m{x}_{e,i}^{k+1}\|^2}{\Gamma_i^{k+1}} \right).
\end{eqnarray*}
Summing this inequality for $k$ between 1 and $t$, with $K = k_0 + k$, yields
\begin{equation} \label{Strong-yyy}
\alpha \sum_{k=1}^t (k_0 +k) \Delta^k 
 + (k_0+t)^2 \theta  \| \m{y}^{t+1} - \m{x}^*\|_{\m{P}}^2 \le \bar{c} ,
\end{equation}
where $\bar{c}$ is defined in (\ref{cbar}).
Substituting for $\Delta^k$ using (\ref{Deltak-form}) and
discarding the $\m{y}^{t+1}$ term, we have
\begin{equation} \label{Strong-zzz}
\alpha \sum_{k=1}^t  (k_0+k) \left[ \C{L}(\m{z}^k, \g{\lambda}^*)
- \Phi(\m{x}^*) \right] \le \bar{c} .
\end{equation}
The convexity of $\Phi$ and
the definition of $\tilde{\m{z}}^k$ in (\ref{def-tilde-z}) yield
\[
\C{L}(\tilde{\m{z}}^k, \g{\lambda}^*))
\le \frac{2}{t(t+1) + 2k_0 t} \sum_{k=1}^t (k_0 + k)
\C{L}(\m{z}^k, \g{\lambda}^*),
\]
which together with (\ref{Strong-zzz}) gives (\ref{acc-rate}). 
In addition, since $\Delta^k \ge 0$,
(\ref{Strong-yyy}) also implies (\ref{y2-conv}).
\end{proof}

As noted at the end of Section~\ref{algorithm},
for either of the parameter choices
(\ref{AG_constant}) or (\ref{AG_linesearch}),
$\gamma^l \ge l^2 \Theta$ for some constant $\Theta > 0$, independent of $k$
and $l$.
Hence, for $l$ sufficiently large, the requirement
(\ref{Strong-Gamma}) at iteration $k+1$ is satisfied.

\section{Linear Convergence}
\label{linearRate}
For the analysis of linear convergence rate of I-ADMM, we assume that
$\psi$ has the additional property that
$\psi(t) \le c_\psi t$ for all $t \ge 0$, where $c_\psi > 0$ is a constant.
Let us define
\begin{equation}\label{errki}
e_i(\m{y}, \g{\lambda}) =
\|\m{y}_i - \mbox{prox}_{h_i}(\m{y}_i - \nabla f_i(\m{y}_i) - \m{A}_i \tr
\g{\lambda} )\|. 
\end{equation}
We begin with the following lemma.
\begin{lemma} \label{linear-lemma}
If the parameters $\delta^l$ and $\alpha^l$
in Algorithm~$\ref{3}$ are chosen
according to either $(\ref{AG_constant})$ or $(\ref{AG_linesearch})$
and $\psi(t) \le c_\psi t$, then for any $k \ge 2$, we have
\begin{equation}\label{upbound-e}
\sum_{i=1}^m e_i(\m{y}^{k+1}, \g{\lambda}^{k+1}) \le c( d_k + d_{k-1}),
\end{equation}
where $c > 0$ is a generic constant which only depends on the problem data
and algorithm parameters such as $\rho$ and $c_\psi$ and
\begin{equation}\label{dk}
d_k = \|\m{y}^k - \m{z}^{k}\| + \|\m{Az}^{k} - \m{b}\| + \sqrt{R^k}.
\end{equation}
\end{lemma}
\begin{proof}
For any $\m{p}_i$ and $\m{q}_i \in \mathbb{R}^{n_i}$, $i = 1, 2,$
it follows from the triangle inequality and the nonexpansive property
of the prox operator that
\begin{eqnarray}
&& \|\m{p}_1 - \mbox{prox}_{h_i} (\m{q}_1)\| \nonumber \\
&=&
\|[\m{p}_2 - \mbox{prox}_{h_i} (\m{q}_2)] +
[\m{p}_1 - \m{p}_2]
+ [\mbox{prox}_{h_i}(\m{q}_2) - \mbox{prox}_{h_i}(\m{q}_1)]\| \nonumber \\
&\le& \|\m{p}_2 - \mbox{prox}_{h_i} (\m{q}_2)\|
+ \|\m{p}_1 - \m{p}_2\| + \|\m{q}_1 - \m{q}_2\| . \label{h89}
\end{eqnarray}
We identify $\|\m{p}_1 - \mbox{prox}_{h_i} (\m{q}_1)\|$ with
$e_i(\m{y}^{k+1}, \g{\lambda}^{k+1})$ and
$\|\m{p}_2 - \mbox{prox}_{h_i} (\m{q}_2)\|$ with
$e_i(\m{z}^{k}, \g{\lambda}^{k})$, and use (\ref{h89}) to obtain
the following bound for
$e_i(\m{y}^{k+1}, \g{\lambda}^{k+1})$ in terms of
$e_i(\m{z}^{k}, \g{\lambda}^{k})$:
\[
e_i(\m{y}^{k+1}, \g{\lambda}^{k+1}) \le
e_i(\m{z}^{k}, \g{\lambda}^{k})
+ (2 + \zeta_i)\|\m{y}_i^{k+1} - \m{z}_i^k\|
+ \|\m{A}_i\tr(\g{\lambda}^{k+1} - \g{\lambda}^k)\|,
\]
where $\zeta_i$ is the Lipschitz constant for $\nabla f_i$.
The update formula for $\g{\lambda}^{k+1}$ implies that
$\g{\lambda}^{k+1} - \g{\lambda}^k =$
$\alpha \rho (\m{A} \m{z}^k - \m{b}) =$
$\alpha \rho \m{r}_k$, where
$\m{r}_k = \m{A} \m{z}^k - \m{b}$.
With this substitution,
the bound for $e_i(\m{y}^{k+1}, \g{\lambda}_i^{k+1})$ becomes
\begin{equation}\label{1qaz}
e_i(\m{y}^{k+1}, \g{\lambda}^{k+1}) \le
e_i(\m{z}^{k}, \g{\lambda}^{k})
+ (2 + \zeta_i)\|\m{y}_i^{k+1} - \m{z}_i^k\|
+ \alpha \rho \|\m{A}_i\tr\m{r}^k\|.
\end{equation}
Let $\nu_i >0$ denote the smallest eigenvalue of $\m{Q}_i$.
The analysis is partitioned into two cases:

{\bf Case 1.} $\Gamma_i^k > 4/(\rho \nu_i)$.
Again, by property (\ref{h89}), we have
\begin{equation}\label{2wsx}
e_i(\m{z}^k, \g{\lambda}^k) \le
e_i(\bar{\m{x}}^k, \g{\lambda}^k)
+ (2 + \zeta_i)\|\m{z}_i^k - \bar{\m{x}}_i^k\|,
\end{equation}
where $\bar{\m{x}}^k$ is given in (\ref{xbar}).
The first-order optimality conditions for $\bar{\m{x}}_i^k$ can be written
\[
\bar{\m{x}}_i^k =
\mbox{prox}_{h_i}\left(\bar{\m{x}}_i^k - \nabla f_i(\bar{\m{x}}_i^k)
- \rho \m{A}_i \tr (\m{A}_i \m{y}_i^k - \m{b}_i^k + \g{\lambda}^k/\rho) 
- \rho \m{Q}_i (\bar{\m{x}}_i^k - \m{y}_i^k) \right).
\]
Using this formula for the first $\bar{\m{x}}_i^k$ on the right side
of the identity
\[
e_i(\bar{\m{x}}^k, \g{\lambda}) =
\|\bar{\m{x}}_i^k - \mbox{prox}_{h_i}(\bar{\m{x}}_i^k
- \nabla f_i(\bar{\m{x}}_i^k) - \m{A}_i \tr \g{\lambda} )\|,
\]
along with the nonexpansive property of prox operator, we have
\[
e_i(\bar{\m{x}}^k, \g{\lambda}^k) \le
\rho \left( \| \m{A}_i \tr (\m{A}_i \m{y}_i^k - \m{b}_i^k) \| +
\|\m{Q}_i (\bar{\m{x}}_i^k - \m{y}_i^k)\| \right) .
\]
The definition of $\m{b}_i^k$ yields
\begin{eqnarray*}
\m{A}_i \m{y}_i^k - \m{b}_i^k &=& \sum_{j < i} \m{A}_j \m{z}_j^k 
+ \sum_{j \ge i} \m{A}_j \m{y}_j^k - \m{b} \\
&=& \m{A} \m{z}^k - \m{b} + \sum_{j \ge i} \m{A}_j( \m{y}_j^k - \m{z}_j^k) \\
&=&  \m{r}_k  + \sum_{j \ge i} \m{A}_j (\m{y}_j^k - \m{z}_j^k) .
\end{eqnarray*}
It follows that
\begin{equation}\label{h92}
\| \m{A}_i \tr (\m{A}_i \m{y}_i^k - \m{b}_i^k) \| \le
c(\|\m{r}_k\| + \|\m{y}^k - \m{z}^k\|),
\end{equation}
and
\begin{equation} \label{eiz}
e_i(\bar{\m{x}}^k, \g{\lambda}^k) \le
c(\|\m{r}_k\| + \|\m{y}^k - \m{z}^k\| + \|\bar{\m{x}}_i^k - \m{z}_i^k\|) .
\end{equation}
Combining this with (\ref{2wsx}) gives
\[
e_i(\m{z}^k, \g{\lambda}^k)  \le
c(\|\m{r}_k\| + \|\m{y}^k - \m{z}^k\| + \|\bar{\m{x}}_i^k - \m{z}_i^k\|) .
\]

Now, by Lemma \ref{lem-AG-Conv}, we have
\begin{equation}\label{h400}
\sqrt{\rho \nu_i} \|\m{z}_i^k- \bar{\m{x}}_i^k\| \le 
\frac{\|\m{x}_i^k- \bar{\m{x}}_i^k\|}{\sqrt{\Gamma_i^k}} \\
\le  \frac{ \|\m{x}_i^k-\m{z}_i^k\|+\|\m{z}_i^k -  \bar{\m{x}}_i^k\|}{\sqrt{\Gamma_i^k}}.
\end{equation}
The stopping condition in Step~1b gives
\begin{equation} \label{stop-cond}
\frac{ \|\m{x}_i^k-\m{z}_i^k\|}{\sqrt{\Gamma_i^k}} \le \psi(\epsilon^{k-1})
\le c \epsilon^{k-1}.
\end{equation}
Hence, by (\ref{h400}) we have
\[
\left( \frac{-1 + \sqrt{\Gamma_i^{k} \rho \nu_i}} {\sqrt{\Gamma_i^{k}}} \right)
\|\m{z}_i^k- \bar{\m{x}}_i^k\|
\le  \frac{ \|\m{x}_i^k-\m{z}_i^k\|}{\sqrt{\Gamma_i^k}} 
\le  c \epsilon^{k-1}.
\]
Therefore, the Case 1 condition $\Gamma_i^{k} > 4/ (\rho \nu_i)$ implies that
\[
\|\m{z}_i^k- \bar{\m{x}}_i^k\| \le  c\epsilon^{k-1},
\]
and by (\ref{eiz}), we have
\begin{equation}\label{zi-sq-1}
e_i(\m{z}^k, \g{\lambda}^k) \le
c(\epsilon^{k-1} + \|\m{y}^k - \m{z}^k\| + \|\m{r}_k\|) .
\end{equation}

{\bf Case 2.} $\Gamma_i^k \le 4/(\rho \nu_i)$.
It is shown in \cite[pp.~227--228]{HagerZhang19} that when the
parameters $\delta^l$ and $\alpha^l$ are chosen
according to either $(\ref{AG_constant})$ or $(\ref{AG_linesearch})$,
there exists a constant $\Theta > 0$, independent of $k$ and $l$, such that
$\gamma^l \ge l^2 \Theta$.
Since the $\gamma^l$ are increasing functions of $l$ and $\Gamma_i^k$
is the final value of $\gamma^l$ in Step~1, it follows from the uniform
bound on $\Gamma_i^k$ in Case 2, and the quadratic growth in $\gamma^l$,
that the final $l$ value in Step 1, which we denote $l_i^k$, is uniformly
bounded as a function of $i$ and $k$.
Also, it follows from the quadratic growth of $\gamma^l$ and equations
(5.18) and (5.20) in \cite{HagerZhang19} that $\delta^l$ is uniformly
(in $k$, $l$, and $i$) bounded.

By the definition of $\gamma^l$ in Algorithm~\ref{3},
we have $(1-\alpha^l)\gamma^l = \gamma^{l-1}$, or equivalently,
$\alpha^l \gamma^l = \gamma^l - \gamma^{l-1}$ (with the convention that
$\gamma^0 = 0$).
Summing this identity over $l$ yields
\begin{equation}\label{gammal}
\gamma^l = \sum_{j=1}^l \alpha^j \gamma^j .
\end{equation}
Next, we multiply the definition
$\m{a}_{ik}^j = (1-\alpha^j) \m{a}_{ik}^{j-1} + \alpha^j \m{u}_{ik}^j$
by $\gamma^j$ and sum over $j$ between 1 and $l$.
Again, exploiting the identity
$(1-\alpha^j) \gamma^j = \gamma^{j-1}$ yields
\begin{equation}\label{convex1}
\m{a}_{ik}^l =
\frac{1}{\gamma^l} \sum_{j=1}^{l} (\gamma^j \alpha^j) \m{u}_{ik}^j. 
\end{equation}
It follows from (\ref{gammal}), that $\m{a}_{ik}^l$ is a convex combination
of $\m{u}_{ik}^j$, $1 \le j \le l$.
If $p_{ik}^j \in [0,1]$ denotes the coefficients in the convex combination,
we have
\begin{equation}\label{cov-a}
\m{a}_{ik}^l = \sum_{j = 1}^l p_{ik}^j \m{u}_{ik}^j,
\end{equation}
Since $\m{z}_i^k = \m{a}_{ik}^L$ for $L = l_i^k$, Jensen's inequality gives
\begin{eqnarray}
e_i(\m{z}^k, \g{\lambda}^k) &\le&
 \sum_{l = 1}^{l_i^k} p_{ik}^l 
\| \m{u}_{ik}^l - \mbox{prox}_{h_i}(\m{z}_i^k
- \nabla f_i(\m{z}_i^k) - \m{A}_i \tr \g{\lambda}^k )\| \nonumber \\
&\le& \sum_{l = 1}^{l_i^k}
\| \m{u}_{ik}^l - \mbox{prox}_{h_i}(\m{z}_i^k
- \nabla f_i(\m{z}_i^k) - \m{A}_i \tr \g{\lambda}^k )\| .
\label{4rfv}
\end{eqnarray}

Now, by the formula for $\m{u}_{ik}^l$ in Alg.~\ref{3}, we have
$\m{u}_{ik}^l = \prox_{h_i} (\m{q}_2)$, where
\[
\m{q}_2 =
\m{u}_{ik}^l - \nabla f_i (\bar{\m{a}}_{ik}^l) - \delta_{ik}^l (\m{u}_{ik}^l
- \m{u}_{ik}^{l-1} ) - \rho \m{A}_i\tr (\m{A}_i \m{y}_i^k - \m{b}_i^k
+ \g{\lambda}^k/\rho)
- \rho \m{Q}_i(\m{u}_{ik}^l - \m{y}_i^k ).
\]
We utilize (\ref{h89}) with
$\m{q}_1 = \m{z}_i^k - \nabla f_i(\m{z}_i^k) - \m{A}_i \tr \g{\lambda}^k$,
with $\m{q}_2$ as given above, and with
$\m{p}_1 = \m{p}_2 = \m{u}_{ik}^l$.
Hence, $\m{p}_2 - \mbox{prox}_{h_i} (\m{q}_{2}) = \m{0}$ and by (\ref{h89}),
it follows that
\begin{eqnarray}
\label{h91} &\| \m{u}_{ik}^l - \mbox{prox}_{h_i}(\m{z}_i^k
- \nabla f_i(\m{z}_i^k) - \m{A}_i \tr \g{\lambda}^k )\| \le
\\
&c \left( \| \m{u}_{ik}^l  -  \m{z}_i^k\| + \|\bar{\m{a}}_{ik}^l
-  \m{z}_i^k \| + \|\m{u}_{ik}^l - \m{u}_{ik}^{l-1} \|
+ \|\m{A}_i\tr (\m{A}_i \m{y}_i^k - \m{b}_i^k)\|
+ \|\m{u}_{ik}^l - \m{y}_i^k\|\right) \le& \nonumber \\
& c \left( \| \m{u}_{ik}^l - \m{z}_i^k\| + \|\bar{\m{a}}_{ik}^l
-  \m{z}_i^k \| + \|\m{u}_{ik}^l - \m{u}_{ik}^{l-1} \|
+ \|\m{A}_i\tr (\m{A}_i \m{y}_i^k - \m{b}_i^k)\|
+ \|\m{y}_i^k - \m{z}_i^k\| \right) &\nonumber
\end{eqnarray}
Each of the terms on the right side of (\ref{h91}) is now analyzed.

Based on (\ref{h92}), the trailing two terms in (\ref{h91}) have the bound
\[
\|\m{A}_i\tr (\m{A}_i \m{y}_i^k - \m{b}_i^k)\| +
\|\m{y}_i^k - \m{z}_i^k\| \le c(\|\m{r}_k\| + \|\m{y}^k - \m{z}^k\|).
\]
The remaining terms in (\ref{h91}) are bounded by $c\sqrt{r_i^k}$ as
will now be shown.
The bound
$\|\m{u}_{ik}^l - \m{u}_{ik}^{l-1} \| \le c\sqrt{r_i^k}$ is a trivial
consequence of the definition of $r_i^k$
and the uniform bound on $\Gamma_i^k$ in Case~2.
By the definition $\bar{\m{a}}_{ik}^l =$
$(1-\alpha^l) (\m{a}_{ik}^{l-1} - \m{u}_{ik}^{l-1}) + \m{u}_{ik}^{l-1}$,
it follows that
\[
\|\bar{\m{a}}_{ik}^l - \m{z}_i^k\| \le
\|\m{a}_{ik}^{l-1} - \m{u}_{ik}^{l-1}\| + \|\m{u}_{ik}^{l-1} - \m{z}_i^k\|.
\]
This inequality and the fact that
$\m{z}_i^k = \m{a}_{ik}^l$ for $l = l_i^k$ implies that all the
remaining terms in (\ref{h91}) have the form
$\|\m{a}_{ik}^l - \m{u}_{ik}^t\|$ for some $l \in [1, l_i^k]$ and some
$t \in [1, l]$.
Combine (\ref{cov-a}), Jensen's inequality,
the fact that $l \le l_i^k$ where $l_i^k$ is uniformly
bounded in Case~2, and the Schwarz inequality to obtain
\[
\| \m{a}_{ik}^l - \m{u}_{ik}^t\| \le
\sum_{j = 1}^l  \left\| \m{u}_{ik}^j - \m{u}_{ik}^t \right\| \le
l \sum_{j = 1}^l \left\| \m{u}_{ik}^j - \m{u}_{ik}^{j-1} \right\|
\le c \sqrt{r_i^k},
\]
These bounds for the terms in (\ref{h91}) combine to yield
\[
\| \m{u}_{ik}^l - \mbox{prox}_{h_i}(\m{z}_i^k
- \nabla f_i(\m{z}_i^k) - \m{A}_i \tr \g{\lambda}^k )\| \le
c\left( \|\m{r}_k\| + \|\m{y}^k - \m{z}^k\| + \sqrt{r_i^k} \right) .
\]
Moreover, by (\ref{4rfv}) and the Case~2 uniform bound on $l_i^k$, we have
\[
e_i(\m{z}^k, \g{\lambda}^k) \le
c \left( \|\m{r}_k\| + \|\m{y}^k - \m{z}^k\| + \sqrt{r_i^k} \right) .
\]
Combine this with the Case~1 lower bound (\ref{zi-sq-1}) gives
\begin{equation}\label{eizk}
e_i(\m{z}^k, \g{\lambda}^k) \le
c \left( \epsilon^{k-1} + \|\m{r}_k\| + \|\m{y}^k - \m{z}^k\|
+ \sqrt{r_i^k} \right) .
\end{equation}
Inserting this in (\ref{1qaz}) yields
\[
e_i(\m{y}^{k+1}, \g{\lambda}^{k+1}) \le
c \left( \epsilon^{k-1} + \|\m{r}_k\| + \|\m{y}^k - \m{z}^k\|
+ \sqrt{r_i^k} + \|\m{y}^{k+1} - \m{y}^k\| \right) .
\]
Based on the back substitution formula
$\m{y}^{k+1} - \m{y}^k = \alpha \m{M}^{- \sf T} \m{Q} (\m{z}^k- \m{y}^k)$,
this reduces to
\[
e_i(\m{y}^{k+1}, \g{\lambda}^{k+1}) \le
c \left( \epsilon^{k-1} + \|\m{r}_k\| + \|\m{y}^k - \m{z}^k\|
+ \sqrt{r_i^k} \right) .
\]
Since $\epsilon^{k-1} \le c d_{k-1}$ and
$\|\m{r}_k\| + \|\m{y}^k - \m{z}^k\| + \sqrt{r_i^k} \le d_k$,
the proof is complete.
\end{proof}

The expression $E_k$ defined in (\ref{def-Ek}) measures the energy
between the current iterate $(\m{x}_k, \m{y}_k, \g{\lambda}_k)$
and a given $(\m{x}^*, \m{x}^*, \g{\lambda}^*)$.
Let $E_k^*$ denote the minimum energy between the iterate
and all possible $(\m{x}^*, \g{\lambda}^*) \in \C{W}^*$.
We will show that when an error bound condition holds,
there exists a constant $\kappa < 1$ such that $E_{k+2}^* \le \kappa E_k^*$.

The error bound condition relates the KKT error to the Euclidean distance
to $\C{W}^*$.
The KKT error $K$ is given by
\begin{equation}\label{errk}
K(\m{x}, \g{\lambda})
= \|\m{A} \m{x} - \m{b}\| + \sum_{i=1}^m e_i(\m{x}, \g{\lambda}).
\end{equation}
When $K(\m{x}, \g{\lambda}) = 0$, the first-order optimality conditions hold.
The Euclidean distance from $(\m{x}, \g{\lambda})$ to $\C{W}^*$ will be
measured by
\begin{equation}\label{E}
\C{E}(\m{x}, \g{\lambda}) =
\min
\left\{\rho \|\m{x} -  \m{x}^* \|^2_{\m{P}}
+ \frac{1}{\rho} \|\g{\lambda} - \g{\lambda}^* \|^2 :
(\m{x}^*, \g{\lambda}^*) \in \C{W}^* \right\}^{1/2} .
\end{equation}
Note that $\m{P} = \m{MQ}^{-1}\m{M}\tr$ is positive definite since
$\m{M}$ is invertible.
Also, by \cite[Prop.~6.1.2]{Bertsekas2003}, every solution of (\ref{Prob})
has exactly the same set of Lagrange multipliers.
If $\m{X}^*$ and $\g{\Lambda}^*$ denote the set of solutions and multipliers
for (\ref{Prob}), then $\C{W}^* = \m{X}^* \times \g{\Lambda}^*$
is a closed, convex set, and there exists a unique
$(\tilde{\m{x}}, \tilde{\g{\lambda}}) \in \C{W}^*$ that achieves the
minimum in (\ref{E}).
The local error bound assumption is as follows:
\smallskip

\begin{assumption}\label{ass-error-bound}
There exist constants $\beta > 0$ and $\eta >0$ such that
$\C{E}(\m{x}, \g{\lambda}) \le \eta K(\m{x}, \g{\lambda})$
whenever $\C{E}(\m{x}, \g{\lambda}) \le \beta$.
\end{assumption}
\smallskip

The local error bound condition is equivalent to saying that in
a neighborhood of $\C{W}^*$, the Euclidean distance to $\C{W}^*$ is bound by
the KKT error, which is often used to analyze the linear convergence behavior
of an optimization algorithm. More recently, a partial error bound condition
based on the ADMM iterates instead of conditions
on the optimization problem is proposed in \cite{LiuYuanZengZhang2018}.
Under such conditions, linear convergence is also established for a 2-block
ADMM.

A multivalued mapping $F$ is piecewise polyhedral if its
graph $\mbox{Gph } F :=\{(\m{x}, \m{y}) : \m{y} \in F(\m{x})\}$
is a union of finitely many polyhedral sets. 
The local error bound condition (Assumption~\ref{ass-error-bound}) holds when 
$\nabla f_i$ is affine and $\partial h_i$ is piecewise polyhedral for
$i=1,\ldots,m$ \cite{HanSunZhang17, Robinson81, YangHan16}.
Note that when $(\m{x}, \g{\lambda})$ is restricted to a bounded set,
the requirement that $\C{E}(\m{x}, \g{\lambda}) \le \beta$ can be dropped.
That is, when $\C{E}(\m{x},\g{\lambda}) > \beta$,
$K(\m{x},\g{\lambda})$ is strictly positive,
and by taking the constant $\eta$ large enough, the bound
$\C{E}(\m{x}, \g{\lambda}) \le \eta K(\m{x}, \g{\lambda})$ holds over the
entire set.
In our analysis, the error bound condition is applied to the iterates
$(\m{y}^{k}, \g{\lambda}^k)$ which lie in a bounded set by
Lemma~\ref{L-key-lemma3}, so the
requirement that $\C{E}(\m{x}, \g{\lambda}) \le \beta$ is unnecessary.

\begin{theorem} \label{first-theorem}
If the parameters $\delta^l$ and $\alpha^l$
in Algorithm~$\ref{3}$ are chosen
according to either $(\ref{AG_constant})$ or $(\ref{AG_linesearch})$,
$\psi (t) \le c_{\psi} t$,
and Assumption~$\ref{ass-error-bound}$ holds,
then there exists $\kappa < 1$ such that
$E_{k+2}^* \le \kappa E_k^*$ at every iteration of Algorithm~$\ref{ADMMcommon}$.
\end{theorem}
\begin{proof}
Let $(\tilde{\m{y}}^{k+1}, \tilde{\g{\lambda}}^{k+1}) \in \C{W}^*$
be the unique minimizer in (\ref{E}) corresponding to
$(\m{x}, \g{\lambda}) =$ $({\m{y}}^{k+1}, {\g{\lambda}}^{k+1})$.
By the stopping condition in Step~1b of Algorithm~\ref{3}, and the definition
of $\Gamma_i^k$ in Step~1c, the sequence
$\Gamma_i^{k}$ is nondecreasing in $k$ by Remark~\ref{stop_condition}.
Since $\Gamma_i^{k}$ is nondecreasing in $k$,
it follows from the triangle inequality and the back substitution formula
$\m{y}^{k+1} - \m{y}^k = \alpha \m{M}^{- \sf T} \m{Q} (\m{z}^k- \m{y}^k)$
that for any $i \in [1,m]$, we have
\begin{eqnarray}
\frac{\|\m{x}_i^{k+1}-\tilde{\m{y}}_i^{k+1}\|}{\sqrt{\Gamma_i^{k+1}}} &\le&
\frac{ \| \m{x}_i^{k+1} - \m{z}_i^k\| + \| \m{z}_i^k - \m{y}_i^k\| 
+ \|\m{y}_i^k - \m{y}_i^{k+1}\|
+ \|\m{y}_i^{k+1}-\tilde{\m{y}}_i^{k+1}\|}{\sqrt{\Gamma_i^{k+1}}}
\nonumber \\
&\le& \frac{\| \m{x}_i^{k+1} - \m{z}_i^k\|}{\sqrt{\Gamma_i^k}}
+  \frac{ \| \m{z}_i^k - \m{y}_i^k\| 
+ \|\m{y}_i^k - \m{y}_i^{k+1}\|
+ \|\m{y}_i^{k+1}-\tilde{\m{y}}_i^{k+1}\|}{\sqrt{\Gamma_i^1}}
\nonumber \\
&\le& \frac{\| \m{x}_i^{k+1} - \m{z}_i^k\|}{\sqrt{\Gamma_i^k}}
+ c \left( \|\m{z}^k - \m{y}^k\| + \|\m{y}_i^{k+1}-\tilde{\m{y}}_i^{k+1}\|
\right),
\label{h100}
\end{eqnarray}
where $c >0$ denotes a generic constant, independent of $k$.

As noted earlier, when the parameters $\delta^l$ and $\alpha^l$
in Algorithm~$\ref{3}$ are chosen
according to either $(\ref{AG_constant})$ or $(\ref{AG_linesearch})$,
we have $\xi^l = \delta^l \alpha^l \gamma^l = 1$.
By equation (\ref{barlik}) with $L = l_i^k$, $\m{u} = \m{a}_i^L = \m{z}_i^k$,
$\m{u}_i^L = \m{x}^{k+1}$, and $\m{u}_i^0 = \m{x}_k$, we obtain the relation
\[
\frac{\|\m{z}_i^k - \m{x}_i^{k+1}\|}{\sqrt{\Gamma_i^k}} \le
\frac{\|\m{z}_i^k - \m{x}_i^k\|}{\sqrt{\Gamma_i^k}} \le \psi(\epsilon^{k-1}),
\]
where the last inequality is due to the stopping condition in Step~1b.
Combining this with (\ref{h100}) yields
\begin{equation}\label{h101}
\frac{\|\m{x}_i^{k+1}-\tilde{\m{y}}_i^{k+1}\|}{\sqrt{\Gamma_i^{k+1}}} \le
\psi(\epsilon^{k-1}) +
c \left( \|\m{z}^k - \m{y}^k\| + \|\m{y}_i^{k+1}-\tilde{\m{y}}_i^{k+1}\|
\right) .
\end{equation}

Exploiting the error bound condition,
we have
\begin{eqnarray}
\|\m{y}^{k+1}-\tilde{\m{y}}^{k+1}\|^2 
&\le& \sqrt{\|\m{P}^{-1}\|} \|\m{y}^{k+1} -  \tilde{\m{y}}^{k+1} \|_{\m{P}}
\label{h102} \\
&\le& c\C{E}(\m{y}^{k+1}, \g{\lambda}^{k+1}) \le
cK (\m{y}^{k+1}, \g{\lambda}^{k+1}). \nonumber
\end{eqnarray}
The constraint violation term in $K$ is estimated as follows:
\[
\|\m{Ay}^{k+1} - \m{b}\| \le \|\m{A}\|(\|\m{y}^{k+1} - \m{y}^{k}\| +
\|\m{y}^k - \m{z}^k\|) + \|\m{Az}^k - \m{b}\| \le c d_k,
\]
where the last inequality is
due to the back substitution formula and the definition (\ref{dk})
of $d_k$.
Hence, Lemma~\ref{linear-lemma} yields
\begin{equation}\label{h103}
K (\m{y}^{k+1}, \g{\lambda}^{k+1}) \le c(d_k + d_{k-1}).
\end{equation}

Combine (\ref{h101})--(\ref{h103}) to obtain
\begin{equation}\label{h104}
\frac{\|\m{x}_i^{k+1}-\tilde{\m{y}}_i^{k+1}\|}{\sqrt{\Gamma_i^{k+1}}} \le
\psi(\epsilon^{k-1}) + c(d_k + d_{k-1}) \le c(d_k + d_{k-1})
\end{equation}
since $\psi(t) \le c_{\psi} t$ and $\epsilon^{k-1} \le cd_{k-1}$.
Since the energy $E_{k+1}^*$ corresponds to the minimum of $E_{k+1}$ over all
$(\m{x}^*, \g{\lambda}^*) \in \C{W}^*$ and since
$(\tilde{\m{y}}^{k+1}, \tilde{\g{\lambda}}^{k+1}) \in \C{W}^*$, it follows that
\[
E_{k+1}^* \le
\rho \|\m{y}^{k+1} -  \tilde{\m{y}}^{k+1} \|^2_{\m{P}}
+ \frac{1}{\rho} \|\g{\lambda}^{k+1} - \tilde{\g{\lambda}}^{k+1} \|^2 + 
\alpha \sum_{i=1}^m \frac{\|\m{x}_i^{k+1} -
\tilde{\m{y}}_i^{k+1} \|^2}{\Gamma_i^{k+1}} .
\]
The first two terms on the right are $\C{E}^2(\m{y}^{k+1}, \g{\lambda}^{k+1})$,
while the last term in bounded by (\ref{h104}).
We have
\[
E_{k+1}^* \le
\C{E}^2(\m{y}^{k+1}, \g{\lambda}^{k+1})
+ c \left(d_k + d_{k-1} \right)^2 .
\]
Combine this with the error bound condition and (\ref{h103}) gives
\begin{equation}\label{Ek+1*}
E_{k+1}^* \le c \left( d_k + d_{k-1} \right)^2 .
\end{equation}

Suppose that
$(\hat{\m{x}}^k, \hat{\g{\lambda}}^k) \in \C{W}^*$ is the unique minimizing
$(\m{x}^*, \g{\lambda}^*) \in \C{W}^*$ associated with $E_k^*$.
By Lemma~$\ref{L-key-lemma3}$ and the fact that
$(\hat{\m{x}}^k, \hat{\g{\lambda}}^k) \in \C{W}^*$,
we have
\begin{eqnarray*}
E_k^* &\ge& \rho \|\m{y}^{k+1} -  \hat{\m{x}}^k \|^2_{\m{P}}
+ \frac{1}{\rho} \|\g{\lambda}^{k+1} - \hat{\g{\lambda}}^k \|^2 + 
\alpha \sum_{i=1}^m \frac{\|\m{x}_i^{k+1} - \hat{\m{x}}_i^k \|^2}{\Gamma_i^k} \\
&&  + \rho\alpha(1-\alpha)
(\|\m{y}^k - \m{z}^{k}\|_{\m{Q}}^2 + \|\m{Az}^{k} - \m{b}\|^2)
+ \sigma \alpha \sum_{i=1}^m  R^k .
\end{eqnarray*}
The first three terms on the right side are bounded from below by $E_{k+1}^*$,
while the last three terms are bounded from below by $c d_k^2$
by the definition of $d_k$ in (\ref{dk}).
Hence,
\begin{equation}\label{Ek*}
E_k^* \ge E_{k+1}^* + c d_k^2.
\end{equation}
We replace $k$ by $k-1$ and then use again
(\ref{Ek*}) followed by (\ref{Ek+1*}) to obtain
\[
E_{k-1}^* \ge E_{k}^* + c d_{k-1}^2 \ge E_{k+1}^* + c(d_k^2 + d_{k-1}^2) \ge
(1+c) E_{k+1}^*,
\]
which completes the proof.
\end{proof}

Another linear convergence result is established when
the objective $\Phi$ is strongly convex, in which case
the solution $\m{x}^*$ of (\ref{Prob}) is unique.
Our assumption is the following:
\begin{assumption}\label{ass-strongconvex}
The objective $\Phi$ is strongly convex with modulus
$\mu > 0$ and there exist constants $\beta > 0$ and $\eta >0$ such that
\begin{equation} \label{errorbound2}
\|\g{\lambda} - \tilde{\g{\lambda}}\| \le \eta
\sum_{i=1}^m \|e_i(\m{x}^*, \g{\lambda})\|
\end{equation}
whenever $\|\g{\lambda} - \tilde{\g{\lambda}}\| \le \beta$.
\end{assumption}
\smallskip

The local error bound condition (\ref{errorbound2}) holds when $\partial h_i$
is piecewise polyhedral for 
$i=1,\ldots,m$ \cite{HanSunZhang17, Robinson81, YangHan16}.
Similar to the comment before Theorem~\ref{first-theorem}, the requirement that
$\|\g{\lambda} - \tilde{\g{\lambda}}\| \le \beta$ can be dropped
since it is applied to the iterates
$\g{\lambda}^k$ which lie in a bounded set by Lemma~\ref{L-key-lemma3}.

\begin{theorem}
If the parameters $\delta^l$ and $\alpha^l$
in Algorithm~$\ref{3}$ are chosen
according to either $(\ref{AG_constant})$ or $(\ref{AG_linesearch})$,
$\psi (t) \le c_{\psi} t$,
and Assumption~$\ref{ass-strongconvex}$ holds,
then there exists $\kappa < 1$ such that
$E_{k+2}^* \le \kappa E_k^*$ at every iteration of Algorithm~$\ref{ADMMcommon}$.
\end{theorem}
\smallskip

\begin{proof}
By the local error bound condition and by (\ref{h89}) with
$\m{p}_1 - \mbox{prox}_{h_i} (\m{q}_1)$ identified with
$e_i(\m{x}^{*}, \g{\lambda}^{k+1})$ and
$\m{p}_2 - \mbox{prox}_{h_i} (\m{q}_2)$ identified with
$e_i(\m{z}^{k}, \g{\lambda}^{k})$, we have
\begin{eqnarray}
\|\g{\lambda}^{k+1} - \tilde{\g{\lambda}}^{k+1}\| &\le&
\eta \sum_{i=1}^m e_i(\m{x}^*, \g{\lambda}^{k+1}) \label{lll-bound} \\
&\le& c \left( \|\m{z}^k - \m{x}^*\|
+ \|\g{\lambda}^{k+1} - \g{\lambda}^k\|
+ \sum_{i=1}^m e_i(\m{z}^k, \g{\lambda}^k) \right), \nonumber
\end{eqnarray}
where $c >0$ is a constant. In the later proof, we again use $c >0$ as a generic constant.
By (\ref{eizk}), it follows that
\[
\sum_{i=1}^m e_i(\m{z}^k, \g{\lambda}^k) \le
c \left( \epsilon^{k-1} + \|\m{r}_k\| + \|\m{y}^k - \m{z}^k\|
+ \sqrt{R^k} \right) .
\]
Inserting this in (\ref{lll-bound}) and recalling that
$\g{\lambda}^{k+1} -\g{\lambda}^k =$ $\alpha \rho ( \m{Az}^{k} - \m{b}) =$
$\alpha\rho \m{r}_k$, we have
\[
\|\g{\lambda}^{k+1} - \tilde{\g{\lambda}}^{k+1}\| \le
c \left( \epsilon^{k-1} + \|\m{z}^k - \m{x}^*\| +
\|\m{r}_k\| + \|\m{y}^k - \m{z}^k\| + \sqrt{R^k} \right) .
\]
Since $\epsilon^{k-1} \le c d_{k-1}$ and
$\|\m{r}_k\| + \|\m{y}^k - \m{z}^k\| + \sqrt{R^k} \le d_k$,
it follows that
\begin{equation}\label{h200}
\|\g{\lambda}^{k+1} - \tilde{\g{\lambda}}^{k+1}\| \le
c( d_k + d_{k-1} + \|\m{z}^k - \m{x}^*\|).
\end{equation}

By (\ref{h101}) with $\tilde{\m{y}}^{k+1} = \m{x}^*$, we have
\begin{equation}\label{h201}
\frac{\|\m{x}_i^{k+1}-\m{x}_i^{*}\|}{\sqrt{\Gamma_i^{k+1}}} \le
c\left( \epsilon^{k-1} +
\|\m{z}^k - \m{y}^k\| + \|\m{y}^{k+1}-{\m{x}}^{*}\| \right) .
\end{equation}
The triangle inequality and the back substitution formula yield
\begin{eqnarray}
\|\m{y}^{k+1} - \m{x}^*\| &\le&
\|\m{y}^{k+1} - \m{y}^k\| + \|\m{y}^{k} - \m{z}^k\| +
\|\m{z}^{k} - \m{x}^*\| \label{h202} \\
&\le& c\|\m{y}^k - \m{z}^k\| + \|\m{z}^{k} - \m{x}^*\| . \nonumber
\end{eqnarray}
The bounds $\epsilon^{k-1} \le cd_{k-1}$ and $\|\m{y}^k - \m{z}^k\| \le d_k$
in  (\ref{h202}) and (\ref{h201}) give
\begin{equation}\label{h500}
\; \quad \quad
\|\m{y}^{k+1} - \m{x}^*\| \le cd_k + \|\m{z}^k - \m{x}^*\|
\mbox{ and }
\frac{\|\m{x}_i^{k+1}-\m{x}_i^{*}\|}{\sqrt{\Gamma_i^{k+1}}} \le
c \left( d_{k-1} + d_k + \|\m{z}^k - \m{x}^*\| \right).
\end{equation}
Combine (\ref{h200}) and (\ref{h500}) to obtain
\begin{eqnarray}
E_{k+1}^* &=&
\rho \|\m{y}^{k+1} -  \m{x}^* \|^2_{\m{P}}
+ \frac{1}{\rho} \|\g{\lambda}^{k+1} - \tilde{\g{\lambda}}^{k+1} \|^2 + 
\alpha \sum_{i=1}^m \frac{\|\m{x}_i^{k+1} - \m{x}_i^* \|^2}{\Gamma_i^{k+1}}
\nonumber \\
&\le& c( d_k + d_{k-1} + \|\m{z}^k - \m{x}^*\|)^2.
\label{1-ok}
\end{eqnarray}
On the other hand, by Lemma~$\ref{L-key-lemma3}$ and the fact that
$(\m{x}^*, \tilde{\g{\lambda}}^k) \in \C{W}^*$, we have
\begin{eqnarray}
E_k^* &\ge& \rho \|\m{y}^{k+1} -  \m{x}^* \|^2_{\m{P}}
+ \frac{1}{\rho} \|\g{\lambda}^{k+1} - \tilde{\g{\lambda}}^k \|^2 + 
\alpha \sum_{i=1}^m \frac{\|\m{x}_i^{k+1} - \m{x}_i^* \|^2}{\Gamma_i^k}
\label{h300}\\
&&  + \rho\alpha(1-\alpha)
(\|\m{y}^k - \m{z}^{k}\|_{\m{Q}}^2 + \|\m{Az}^{k} - \m{b}\|^2)
+ \sigma \alpha R^k + 2 \alpha \Delta^k \nonumber \\
&\ge& E_{k+1}^* + cd_k^2 +  \mu \|\m{z}^k - \m{x}^*\|^2, \nonumber
\end{eqnarray}
where the last inequality is due to the definition (\ref{dk}) of $d_k$ and the
strong convexity of $\Phi$:
\[
\Delta^k := \Phi(\m{z}^k) - \Phi(\m{x}^*)
+ (\tilde{\g{\lambda}}^k, \m{A} \m{z}^k - \m{b})
\ge \frac{\mu}{2} \|\m{z}^k - \m{x}^*\|^2.
\]
Finally, we replace $k$ by $k-1$ in (\ref{h300}),
and then use again (\ref{h300}) followed by (\ref{1-ok}) to obtain
\[
E_{k-1}^* \ge E_{k}^* + c d_{k-1}^2 \ge E_{k+1}^* + c(d_k^2 + d_{k-1}^2)
+ \mu \|\m{z}^k - \m{x}^*\|^2 \ge (1+c) E_{k+1}^*,
\]
which completes the proof.
\end{proof}

\section{Numerical Experiments}
\label{numerical}
In this section, we compare the performance of I-ADMM to that of
two different algorithms: (a)~linearized ADMM with one linearization
step for each subproblem and (b)~exact ADMM where the subproblems
are solved either by the conjugate gradient method or by an explicit formula.
The conjugate gradient method was well suited for the quadratic subproblems
in our test set.
We tried using a small number of conjugate gradient iterations to solve
a subproblem, such as 5 iterations starting from the solution computed in
the previous iteration, but found that the scheme did not converge.
Instead we continued the CG iteration until the norm of the gradient
was at most $10^{-6}$.
The one-step ADMM algorithm that we used in (a) for the experiments was the
generalized BOSVS algorithm from \cite{HagerZhang19}.
This algorithm is globally convergent, and although the penalty term was not
linearized, it was possible to quickly solve the subproblems that arise in the
imaging test problems using a fast Fourier transform, as explained in 
\cite{chy13}.

The problems in our experiments were the same image reconstruction problems
used in \cite{HagerZhang19}.
One image employs a blurred version of the well-known Cameraman
image of size $256 \times 256$,
while the second set of test problems,
which arise in partially parallel imaging (PPI), are found in \cite{chy13}.
The observed PPI data, corresponding to 3 different images, are
denoted data~1, data~2, and data~3.
These image reconstruction problem can be formulated as
\begin{equation}\label{3block-obj}
\min_{\m{u}} \; \frac{1}{2} \|\m{Fu} - \m{f}\|^2
+ \alpha \|\m{u}\|_{TV} + \beta \|\g{\Psi} \tr \m{u}\|_1,
\end{equation}
where $\m{f}$ is the given image data, $\m{F}$ is a matrix describing
the imaging device, $\|\cdot\|_{TV}$ is the total variation norm,
$\| \cdot \|_1$ is the $\ell_1$ norm,
$\g{\Psi}$ is a wavelet transform, and
$\alpha>0$ and $\beta >0$ are weights.
The first term in the objective is the data fidelity term,
while the next two terms are for regularization;
they are designed to enhance edges and increase image sparsity.
In our experiments, $\g{\Psi}$ is a normalized Haar wavelet 
with four levels and $\g{\Psi} \g{\Psi} \tr = I$.
The problem (\ref{3block-obj}) is equivalent to 
\begin{equation}\label{3block-equiv}
\min_{(\m{u}, \m{v}, \m{w})} \; \frac{1}{2} \|\m{Fu} - \m{f}\|^2
+ \alpha \|\m{w}\|_{1,2} + \beta \|\m{v}\|_1 \; \mbox{subject to }
\m{Bu} = \m{w}, \; \g{\Psi} \tr \m{u} = \m{v},
\end{equation}
where $\m{Bu} = \nabla \m{u}$ and $(\nabla \m{u})_i$
is the vector of finite differences in the image along the
coordinate directions at the i-th pixel in the image,
$\|\m{w}\|_{1,2} = \sum_{i=1}^N \| (\nabla \m{u})_i\|_2$,
and $N$ is the total number of pixels in the image.

The problem (\ref{3block-equiv}) has the structure appearing in
(\ref{Prob})--(\ref{ProbM}) with $h_1 := 0$,
$f_1 (\m{u})=$ $1/2  \|\m{Fu} - \m{f}\|^2$,
$ h_2(\m{w}) = \|\m{w}\|_{1,2}$,
$f_2 := 0$,
$h_3(\m{v}) = \|\m{v}\|_1$,
$f_3 := 0$,
\[
\begin{array}{c}
\m{A}_1 = \left( 
\begin{array}{l}
 \m{B} \\
 \g{\Psi} \tr
\end{array}
 \right),
\quad
\m{A}_2 = \left( 
\begin{array}{r}
-\m{I} \\
 \m{0}
\end{array}
\right), \quad
\m{A}_3 =
\left( 
\begin{array}{r}
 \m{0} \\
-\m{I}
\end{array}
\right),
\quad \mbox{and} \quad
\m{b} =
\left( 
\begin{array}{r}
\m{0} \\
\m{0}
\end{array}
\right).
\end{array}
\]
The algorithm parameters
$\alpha^l$ and $\delta^l$ were chosen as in (\ref{AG_linesearch}).
Since $f_2 = f_3 = 0$,
the second and third subproblems are solved in closed form,
due to the simple structure of $h_2$ and $h_3$.
Only the first subproblem is solved inexactly.
At iteration $k$, the solution of this subproblem approximates the
solution of
\[
\min_{\m{u}} \;
 \frac{1}{2} \|\m{Fu} - \m{f}\|^2 +
\frac{\rho}{2} \| \m{Bu} - \m{w}^k + \rho^{-1} \g{\lambda}^k\|^2
+  \frac{\rho}{2} \|  \g{\Psi} \tr \m{u} - \m{v}^k
+ \rho^{-1} \g{\mu}^k\|^2,
\]
where $\g{\lambda}^k$ and $\g{\mu}^k$ are the Lagrange multipliers
at iteration $k$ for the constraints
$\m{Bu} =$ $\m{w}$ and $\g{\Psi}\tr \m{u} =$ $\m{v}$ respectively.
Details of the experimental setup can be found in \cite{HagerZhang19}.
The $i$-th block diagonal element of $\m{Q}$ was taken to be a multiple
$\gamma_i$ of the identity $\m{I}$.
According to the assumptions of IADM, $\gamma_1$ should be chosen
large enough that $\gamma_1 \m{I} - \m{A}_1\tr \m{A}_1$ is positive
semidefinite, where
\[
\m{A}_1\tr \m{A}_1 = \m{B}\tr \m{B} + \g{\Psi} \g{\Psi}\tr.
\]
However, a closer inspection of the global convergence proof reveals
that for convergence, it is sufficient to have
\begin{equation}\label{***}
\gamma_1 \|\m{z}^k - \m{y}^k\|^2 \ge \|\m{A}_1(\m{z}^k - \m{y}^k)\|^2
\end{equation}
in each iteration.
Instead of computing the largest eigenvalue of $\m{A}_1\tr \m{A}_1$,
we simply start with $\gamma_1 = 4$ and multiply it by a constant factor
(3 in the experiments) whenever the inequality (\ref{***}) is violated.
Within a finite number of iterations, $\gamma_1$ is large enough that
(\ref{***}) always holds.

Figure~\ref{error_plots} plots the logarithm of the relative objective error
versus the CPU time for the four test problems and the three methods.
Note that the first few iterations of the exact ADMM for Data~3 have error
greater than one, so they missing from the plot.
Observe that I-ADMM performed better than the exact ADMM and the
exact ADMM was generally better than the single linearization step,
except possibly in the initial iterations where the high accuracy of
the exact ADMM was not helpful.
I-ADMM gave better performance both initially and asymptotically.

\begin{figure}
\centering
{\rm (a)} {\includegraphics[width=.45\textwidth]{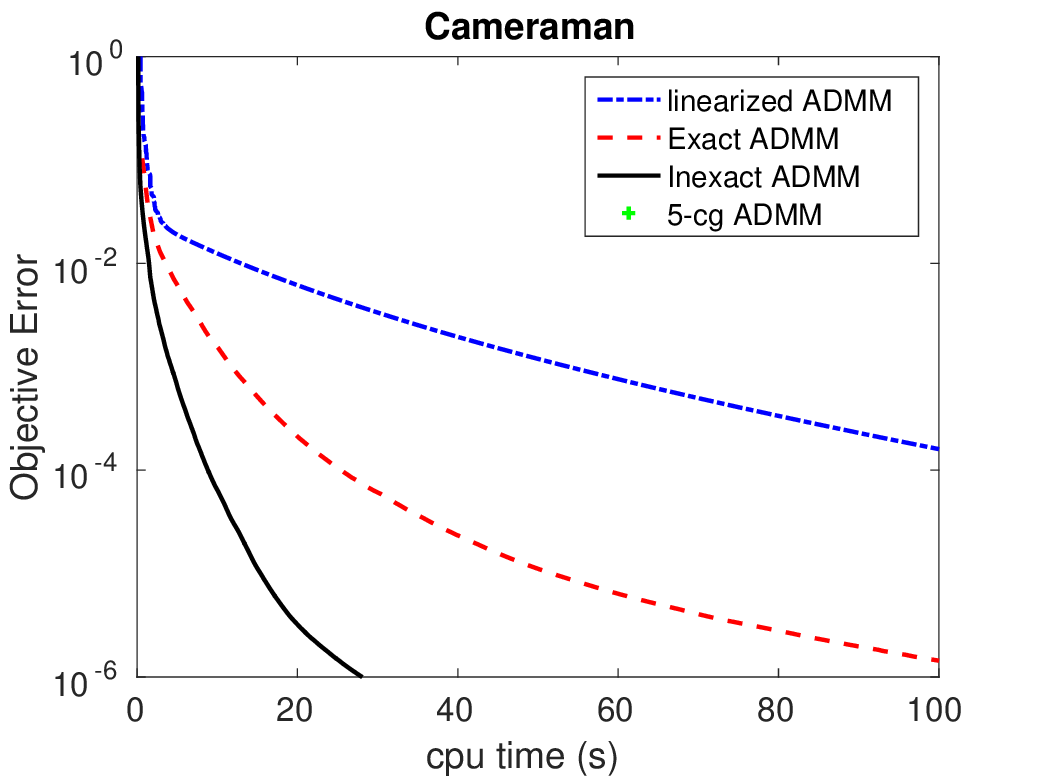}}
{\rm (b)} {\includegraphics[width=.45\textwidth]{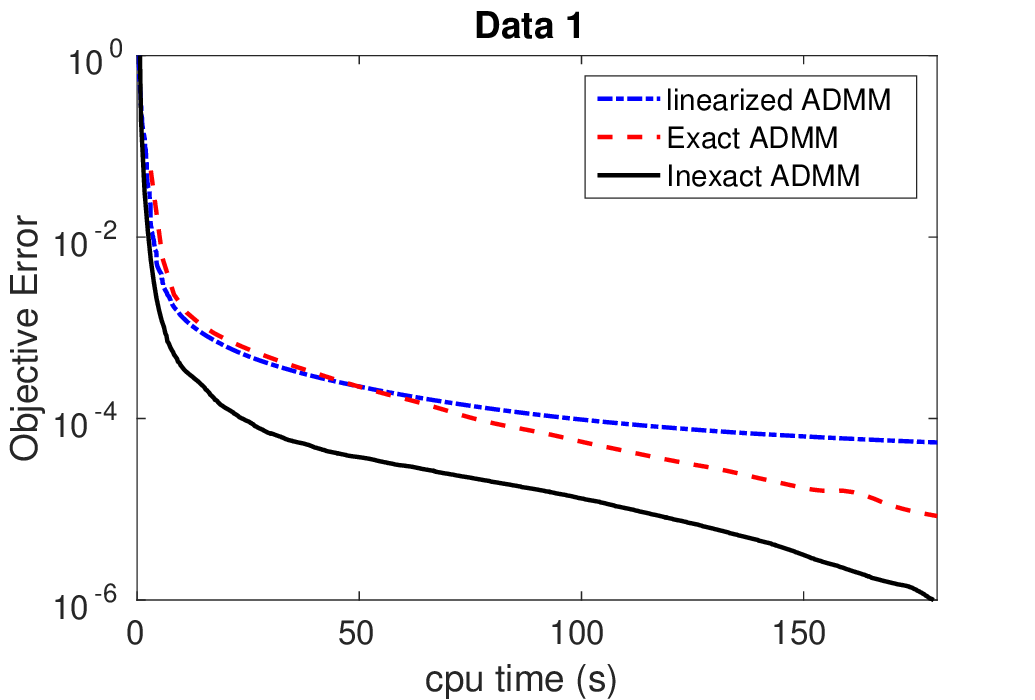}}
\\[.2in]
{\rm (c)} {\includegraphics[width=.45\textwidth]{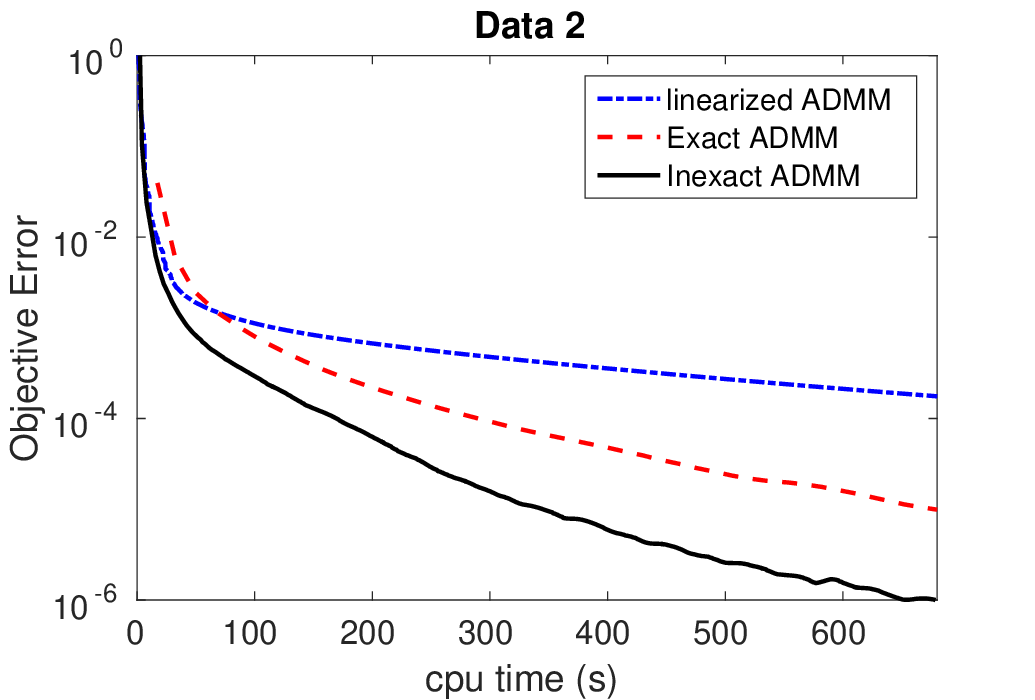}}
{\rm (d)} {\includegraphics[width=.45\textwidth]{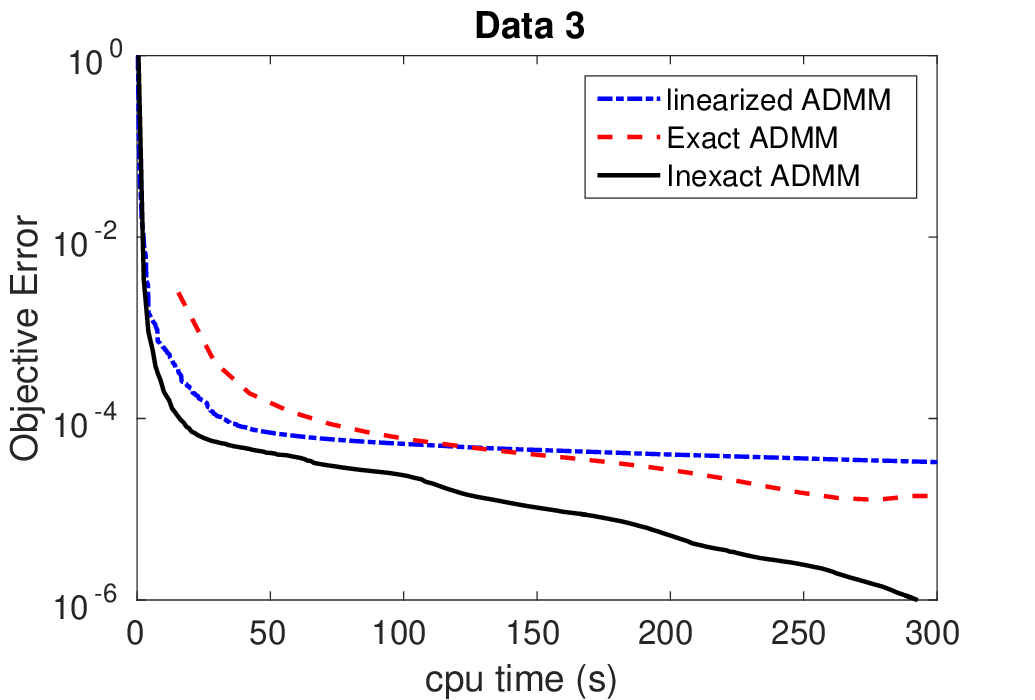}}
\caption{Base-10 logarithm of the relative objective error versus CPU time
for the test problems.
}
\label{error_plots}
\end{figure}

%
\section{Conclusion}
\label{conclusion}
We propose an inexact alternating direction method of multipliers, I-ADMM, for
solving separable convex linearly constrained optimization problems,
where the objective is the sum of smooth and relatively simple nonsmooth terms.
The nonsmooth terms could be infinite, so the algorithms and analysis
include problems with additional convex constraints.
This I-ADMM emanates for our earlier work \cite{chy13, hyz16, HagerZhang19}
on a Bregman Operator Splitting algorithm with a variable stepsize (BOSVS).
The subproblems are solved using an accelerated gradient algorithm
that employs a linearization of both the smooth objective and the penalty term.
We establish an $\C{O}(1/k)$ ergodic convergence rate
for I-ADMM, where $k$ is the iteration number. 
Under a strong convexity assumption, the convergence rate improves to
$\C{O}(1/k^2)$ for both ergodic and nonergodic iterates.
When an error bound condition holds, 2-step linear convergence is
established for nonergodic iterates.
The convergence rates for I-ADMM are
consistent with convergence rates obtained for exact ADMM schemes such as
those in \cite{HanSunZhang17, HeYuan12, HongLuo2017, RMS13,
ShefiTeboulle14, YangHan16}.
As observed in the numerical experiments, an advantage of the inexact scheme
is that the computing time to achieve a given error tolerance is reduced,
when compared to the the exact iteration,
since the accuracy of the subproblem solutions are adaptively increased
as the iterates converge so as to achieve the same convergence rates as
the exact algorithms.
%
\section{Appendix: Proofs for the Global Convergence Analysis}
\label{appendix}
For reference, given a smooth function
$\Psi: \mathbb{R}^n \rightarrow \mathbb{R}$ and a convex
real-valued function $h$ with convexity modulus $\mu$,
the first-order optimality condition for a minimizer $\m{u}$
of the sum $\Psi (\cdot) + h (\cdot)$ is given by
\begin{equation}\label{first-order}
h(\m{u}) + (\mu/2)\|\m{w}-\m{u}\|^2 \le
h (\m{w}) + \nabla \Psi (\m{u})(\m{w} - \m{u})
\end{equation}
for every $\m{w} \in \mathbb{R}^{n}$.

{\bf Proof of Lemma~\ref{lem-AG-Conv}.}
By the definition
$\m{a}_i^l = (1-\alpha^l) \m{a}_i^{l-1} + \alpha^l \m{u}_i^l$, we have
\[
\langle \nabla  f_i({\bar{\m{a}}}_i^l ),
\m{a}_i^l  - {\bar{\m{a}}}_i^l \rangle =
(1 - \alpha^l)
\langle \nabla  f_i({\bar{\m{a}}}_i^l) ,
\m{a}_i^{l-1}  - {\bar{\m{a}}}_i^l \rangle
+ \alpha^l \langle \nabla  f_i({\bar{\m{a}}}_i^l ),
\m{u}_i^l - {\bar{\m{a}}}_i^l \rangle .
\]
Add to this the identity
$f_i({\bar{\m{a}}}_i^l ) = (1-\alpha^l) f_i({\bar{\m{a}}}_i^l ) +
\alpha^l f_i({\bar{\m{a}}}_i^l )$ to obtain
\begin{eqnarray*}
&f_i({\bar{\m{a}}}_i^l ) +
\langle \nabla  f_i({\bar{\m{a}}}_i^l ),
\m{a}_i^l  - {\bar{\m{a}}}_i^l \rangle =& \nonumber \\[.05in]
&(1 - \alpha^l) \left[ f_i({\bar{\m{a}}}_i^l ) +
\langle \nabla  f_i({\bar{\m{a}}}_i^l) ,
\m{a}_i^{l-1}  - {\bar{\m{a}}}_i^l \rangle \right]
+ \alpha^l \left[ f_i({\bar{\m{a}}}_i^l )
+ \langle \nabla  f_i({\bar{\m{a}}}_i^l ),
\m{u}_i^l - {\bar{\m{a}}}_i^l \rangle \right].&
\end{eqnarray*}
By the convexity of $f_i$, it follows that
$f_i({\bar{\m{a}}}_i^l ) +
\langle \nabla  f_i({\bar{\m{a}}}_i^l) ,
\m{a}_i^{l-1}  - {\bar{\m{a}}}_i^l \rangle \le f_i(\m{a}_i^{l-1})$.
Hence, 
\[
f_i({\bar{\m{a}}}_i^l ) +
\langle \nabla  f_i({\bar{\m{a}}}_i^l ) ,
\m{a}_i^l  - {\bar{\m{a}}}_i^l \rangle  \le
(1 - \alpha^l) f_i(\m{a}_i^{l-1}) + 
\alpha^l \left[ f_i({\bar{\m{a}}}_i^l )
+ \langle \nabla  f_i({\bar{\m{a}}}_i^l ),
\m{u}_i^l - {\bar{\m{a}}}_i^l \rangle \right] .
\]
Adding and subtracting any $\m{u} \in \mathbb{R}^{n_i}$ in the last
term, and then exploiting the convexity of $f_i$ gives
\begin{eqnarray*}
f_i({\bar{\m{a}}}_i^l )
+ \langle \nabla  f_i({\bar{\m{a}}}_i^l ),
\m{u}_i^l - {\bar{\m{a}}}_i^l \rangle &=&
\left[ f_i({\bar{\m{a}}}_i^l )
+ \langle \nabla  f_i({\bar{\m{a}}}_i^l ),
\m{u} - {\bar{\m{a}}}_i^l \rangle \right] +
\langle \nabla  f_i({\bar{\m{a}}}_i^l ),
\m{u}_i^l - \m{u} \rangle \\
&\le& f_i(\m{u}) + \langle \nabla  f_i({\bar{\m{a}}}_i^l ),
\m{u}_i^l - \m{u} \rangle .
\end{eqnarray*}
Therefore,
\begin{equation} \label{1234}
\quad \quad f_i({\bar{\m{a}}}_i^l ) +
\langle \nabla  f_i({\bar{\m{a}}}_i^l ),
\m{a}_i^l  - {\bar{\m{a}}}_i^l \rangle \le
(1 - \alpha^l) f_i (\m{a}_i^{l-1} ) +
\alpha^l [ f_i(\m{u}) + \langle \nabla
f_i({\bar{\m{a}}}_i^l ), \m{u}_i^l  - \m{u} \rangle ].
\end{equation}

Now by the line search condition in Step~1a of Algorithm~\ref{3} and
then by (\ref{1234}), we have
\begin{eqnarray*}
L_i^k (\m{a}_i^l) &=& f_i (\m{a}_i^l) +
\frac{\rho}{2} \|\m{A}_i \m{a}_i^l - \m{b}_i^k+
\g{\lambda}^k/\rho\|^2 + h_i  (\m{a}_i^l) \nonumber \\
&\le& f_i({\bar{\m{a}}}_i^l) +
\langle \nabla f_i({\bar{\m{a}}}_i^l), \m{a}_i^l -{\bar{\m{a}}}_i^l \rangle +
\frac{(1-\sigma)\delta^l}{2 \alpha^l} \|\m{a}_i^l - {\bar{\m{a}}}_i^l\|^2
\nonumber
\\[.05in]
&&\quad \quad \quad +
\frac{\rho}{2} \|\m{A}_i \m{a}_i^l - \m{b}_i^k + \g{\lambda}^k/\rho\|^2 +
h_i  (\m{a}_i^l) \nonumber \\[.05in]
&\le&
(1 - \alpha^l) f_i (\m{a}_i^{l-1} ) +
\alpha^l f_i(\m{u}) + \alpha^l \langle \nabla
f_i({\bar{\m{a}}}_i^l ), \m{u}_i^l  - \m{u} \rangle +
\frac{(1-\sigma)\delta^l}{2 \alpha^l} \|\m{a}_i^l - {\bar{\m{a}}}_i^l\|^2
\nonumber
\\[.05in]
&& \quad \quad \quad +
\frac{\rho}{2} \|\m{A}_i \m{a}_i^l - \m{b}_i^k + \g{\lambda}^k/\rho\|^2 +
h_i  (\m{a}_i^l). \label{1235}
\end{eqnarray*}
Next, we utilize the definitions of $\m{a}_i^l$ and $\bar{\m{a}}_i^l$, and
the convexity of both $h_i$ and the norm term to obtain
\begin{eqnarray}
L_i^k (\m{a}_i^l) &\le&
(1 - \alpha^l) f_i (\m{a}_i^{l-1} ) +
\alpha^l [ f_i(\m{u}) +  \langle \nabla  f_i({\bar{\m{a}}}_i^l ), 
\m{u}_i^l  - \m{u} \rangle] +
\frac{(1-\sigma)\delta^l}{2 \alpha^l} \|\m{a}_i^l - {\bar{\m{a}}}_i^l\|^2
\nonumber \\
&&  \quad + (1-\alpha^l) \left(\frac{\rho}{2}
\|\m{A}_i \m{a}_i^{l-1} - \m{b}_i^k  + \g{\lambda}^k/\rho \|^2 
+ h_i  (\m{a}_i^{l-1}) \right) \nonumber \\
&& \quad + \alpha^l \left(\frac{ \rho}{2} \|\m{A}_i \m{u}_i^l - \m{b}_i^k
+ \g{\lambda}^k/\rho \|^2 + h_i  (\m{u}_i^l) \right) \nonumber\\
&=&  (1-\alpha^l) \left( f_i (\m{a}_i^{l-1} ) +
\frac{\rho}{2} \|\m{A}_i \m{a}_i^{l-1} - \m{b}_i^k +
\g{\lambda}^k/\rho \|^2 +  h_i  (\m{a}_i^{l-1})  \right) \nonumber\\ 
&&\quad + \alpha^l [ f_i(\m{u}) + \langle \nabla
 f_i({\bar{\m{a}}}_i^l ),\m{u}_i^l  - \m{u} \rangle]
+ \frac{(1-\sigma)\delta^l \alpha^l}{2}
\|\m{u}_i^l - \m{u}_i^{l-1}\|^2   \nonumber\\ 
&&\quad + \alpha^l \left(\frac{ \rho}{2}
\|\m{A}_i \m{u}_i^l - \m{b}_i^k +
\g{\lambda}^k/\rho \|^2 + h_i  (\m{u}_i^l) \right) \nonumber \\
&=&  (1-\alpha^l) L_i^k (\m{a}_i^{l-1} ) 
+ \alpha^l [ f_i(\m{u}) +
\langle \nabla  f_i({\bar{\m{a}}}_i^l ),\m{u}_i^l  - \m{u} \rangle]
\nonumber \\
&& \quad
+ \frac{(1-\sigma)\delta^l \alpha^l}{2 } \|\m{u}_i^l - \m{u}_i^{l-1}\|^2
+  \alpha^l \left(\frac{ \rho}{2} \|\m{A}_i \m{u}_i^l - \m{b}_i^k
+ \g{\lambda}^k/\rho\|^2 + h_i  (\m{u}_i^l) \right).  \label{qwer}
\end{eqnarray}
By (\ref{first-order}),
the first-order optimality condition for $\m{u}_i^l$ in Step~1a is
\begin{eqnarray}
&&h_i (\m{u}_i^l) + (\mu_i/2) \|\m{u} - \m{u}_i^l\|^2 \label{h2} \\
&\le& \langle \nabla P(\m{u}_i^l), (\m{u} - \m{u}_i^l) \rangle
+ \rho \langle \m{u}_i^l - \m{y}_i^k ,
\bar{\m{Q}}_i(\m{u} - \m{u}_i^l) \rangle + h_i(\m{u}),
\nonumber
\end{eqnarray}
where
\[
\nabla P(\m{u}_i^l) = \nabla f_i (\bar{\m{a}}_i^l)
+\delta^l (\m{u}_i^l - \m{u}_i^{l-1})
+\rho (\m{A}_i\tr (\m{A}_i \m{u}_i^l - \m{b}_i^k + \g{\lambda}^k/\rho).
\]
Multiply (\ref{h2}) by $\alpha^l$ and add to (\ref{qwer}) to obtain
(after some algebra):
\begin{eqnarray*}
&L_i^k(\m{a}_i^l)
+ (\alpha^l \mu_i/2)\|\m{u}-\m{u}_i^l\|^2 \le& \\
&(1-\alpha^l) L_i^k (\m{a}_i^{l-1} ) + \alpha^l
L_i^k(\m{u})  + \frac{ \delta^l \alpha^l }{2}
( \|\m{u} - \m{u}_i^{l-1}\|^2 - \|\m{u} - \m{u}_i^l \|^2)& \\
&-\frac{\sigma\delta^l \alpha^l}{2 } \|\m{u}_i^l - \m{u}_i^{l-1}\|^2  
 - \frac{\alpha^l \rho}{2}\|\m{A}_i  (\m{u} -\m{u}_i^l)\|^2
 + \rho \alpha^l(\m{u}^l - \m{y}_i^k) \tr  \bar{\m{Q}}_i (\m{u} - \m{u}_i^l).&
\end{eqnarray*}
Hence, for any $\m{u} \in \mathbb{R}^{n_i}$ we have
\begin{eqnarray}
&L_i^k (\m{a}_i^l) - L_i^k(\m{u})
+ (\alpha^l \mu_i/2)\|\m{u}-\m{u}_i^l\|^2 \le&
\label{asdf} \\
&(1-\alpha^l) (L_i^k (\m{a}_i^{l-1}) - L_i^k(\m{u}))
+ \frac{ \delta^l \alpha^l }{2} ( \|\m{u} - \m{u}_i^{l-1}\|^2
- \|\m{u} - \m{u}_i^l \|^2)  \nonumber \\
& -\frac{\sigma\delta^l \alpha^l}{2 } \|\m{u}_i^l - \m{u}_i^{l-1}\|^2 
- \frac{\alpha^l \rho}{2}\|\m{A}_i  (\m{u} -\m{u}_i^l)\|^2
+ \rho \alpha^l(\m{u}^l - \m{y}_i^k) \tr  \bar{\m{Q}}_i (\m{u} - \m{u}^l). &
\nonumber
\end{eqnarray}

From the definition of $\gamma^l$ in Algorithm~\ref{3}, it follows that
$(1-\alpha^l) \gamma^l = \gamma^{l-1}$ with the convention that
$\gamma^0 = 0$ (since $\alpha^1 = 1)$.
Hence, for any sequence $d^l$, $l \ge 0$, we have
\begin{equation}\label{id1}
\sum_{l=1}^j \left( \gamma^l d^l - (1-\alpha^l)\gamma^l d^{l-1}
\right) =
\sum_{l=1}^j \left( \gamma^l d^l - \gamma^{l-1} d^{l-1} \right) =
\gamma^j d^j.
\end{equation}
Suppose that $d^l \ge 0$ for each $l$.
By assumption, $\xi^l = \gamma^l \delta^l \alpha^l$ is nonincreasing;
since $\alpha^1 = 1$ and $\gamma^1 = 1/\delta^1$, it follows that
$\xi^1 = 1$, and we have
\begin{eqnarray}
\sum_{l=1}^j \xi^l \left( d^l - d^{l-1} \right) &=&
d^1 - d^0 + \sum_{l=2}^j \xi^l \left( d^l - d^{l-1} \right)
\label{id2} \\
&\ge& d^1 - d^0 + \sum_{l=2}^j \left( \xi^l d^l - \xi^{l-1} d^{l-1} \right) =
\xi^j d^j - d^0 . \nonumber
\end{eqnarray}
We now multiply (\ref{asdf}) by $\gamma^l$ and sum over
$l$ between 1 and $L$.
Exploiting the identity (\ref{id1}) with
$d^l = L_i^k (\m{a}_i^l) - L_i^k(\m{u})$ and
(\ref{id2}) with $d^l = \|\m{u}_i^l - \m{u}\|^2$, we obtain
\begin{eqnarray}
&L_i^k(\m{u}) -  L_i^k (\m{a}_i^{L}) \ge
\frac{1}{2 \gamma^L}
(\xi^{L} \|\m{u} - \m{u}_i^{L} \|^2 -   \|\m{u} - \m{u}_i^{0}\|^2 ) &
\label{zxcv} \\
&
+ \frac{\sigma}{2 \gamma^L} \sum_{l=1}^{L} 
\xi^l  \|\m{u}_i^l - \m{u}_i^{l-1}\|^2
+ \frac{\rho}{2 \gamma^L}
\sum_{l=1}^{L} (\gamma^l \alpha^l) \|\m{A}_i  (\m{u} -\m{u}_i^l)\|^2 &
\nonumber\\
&
+ \frac{\rho}{ \gamma^L}
\sum_{l=1}^{L} (\gamma^l \alpha^l) (\m{u}_i^l - \m{y}_i^k) \tr  \bar{\m{Q}}_i 
(\m{u}_i^l - \m{u})
+ \frac{\mu_i}{2\gamma^L} \sum_{l=1}^L (\gamma^l \alpha^l )
\|\m{u}-\m{u}_i^l\|^2 . &
\nonumber
\end{eqnarray} 

Next, we multiply the definition
$\m{a}_i^j = (1-\alpha^j) \m{a}_i^{j-1} + \alpha^j \m{u}_i^j$
by $\gamma^j$ and sum over $j$ between 1 and $l$.
Again, exploiting the identity
$(1-\alpha^j) \gamma^j = \gamma^{j-1}$ yields
\begin{equation}\label{convex2}
\m{a}_i^l =
\frac{1}{\gamma^l} \sum_{j=1}^{l} (\gamma^j \alpha^j) \m{u}_i^j. 
\end{equation}
Since $\alpha^j \gamma^j = \gamma^j - \gamma^{j-1}$, it follows that
\[
\gamma^l = \sum_{j=1}^l \alpha^j \gamma^j .
\]
Consequently, $\m{a}_i^l$ is a convex combination of
$\m{u}_i^1$ through $\m{u}_i^l$.
Since $\|\m{A}_i  (\m{u} - \m{w})\|^2$,
$(\m{w} - \m{y}_i^k) \tr  \bar{\m{Q}}_i (\m{w} - \m{u}),$
and $\|\m{u}-\m{w}\|^2$ are convex functions of $\m{w}$,
Jensen's inequality can be applied to each of the last three terms in
(\ref{zxcv}).
For example, we have
\[
\frac{1}{\gamma^L}\sum_{l=1}^L (\gamma^l \alpha^l ) \|\m{u}-\m{u}_i^l\|^2 \ge
\|\m{u} - \m{a}_i^L\|^2 .
\]
The net effect of Jensen's inequality is to delete the summation and
replace $\m{u}_i^l$ by $\m{a}_i^L$ in the last three terms of (\ref{zxcv})
to obtain
\begin{eqnarray}
&L_i^k(\m{u}) -  L_i^k (\m{a}_i^{L})
\ge \frac{1}{2 \gamma^L}
(\xi^{L} \|\m{u} - \m{u}_i^{L} \|^2 -   \|\m{u} - \m{u}_i^{0}\|^2 ) &
\label{generalu} \\
&
+ \frac{\rho}{2} \|\m{A}_i  (\m{u} -\m{a}_i^L)\|^2
+ \rho (\m{a}_i^L - \m{y}_i^k) \tr  \bar{\m{Q}}_i (\m{a}_i^L - \m{u})
+ \frac{\mu_i}{2} \|\m{u} - \m{a}_i^L\|^2 
+ \frac{\sigma}{2 \gamma^L} \sum_{l=1}^{L}
\xi^l  \|\m{u}_i^l - \m{u}_i^{l-1}\|^2 .
&
\nonumber
\end{eqnarray}
Hence, after discarding the $\m{u}_i^L$ term, we have
\begin{eqnarray}
\bar{L}_i^k(\m{u}) -  \bar{L}_i^k (\m{a}_i^{L})
&\ge& \frac{-1}{2 \gamma^L} \|\m{u} - \m{u}_i^{0}\|^2
+ \frac{\sigma}{2 \gamma^L} \sum_{l=1}^{L} 
\xi^l  \|\m{u}_i^l - \m{u}_i^{l-1}\|^2
\label{generalu-1} \\
&& + \frac{\rho}{2} \|\m{A}_i  (\m{u} -\m{a}_i^L)\|^2
 + \rho (\m{a}_i^L - \m{y}_i^k) \tr  \bar{\m{Q}}_i (\m{a}_i^L - \m{u})
\nonumber \\
&& + \frac{\rho}{2} \| \m{u} - \m{y}_i^k \|^2_{\bar{\m{Q}}_i}
 - \frac{\rho}{2} \| \m{a}_i^L - \m{y}_i^k \|^2_{\bar{\m{Q}}_i}
+ \frac{\mu_i}{2} \|\m{u} - \m{a}_i^L\|^2 \nonumber \\
&=& \frac{-1}{2 \gamma^L} \|\m{u} - \m{u}_i^{0}\|^2
+ \frac{\sigma}{2 \gamma^L} \sum_{l=1}^{L} 
\xi^l  \|\m{u}_i^l - \m{u}_i^{l-1}\|^2  \nonumber\\
&& + \frac{\rho}{2} \|\m{A}_i  (\m{u} -\m{a}_i^L)\|^2
 + \frac{\rho}{2} \| \m{u} - \m{a}_i^L \|_{\bar{\m{Q}}_i}^2
+ \frac{\mu_i}{2} \|\m{u} - \m{a}_i^L\|^2 \nonumber \\
&=& \frac{-1}{2 \gamma^L} \|\m{u} - \m{u}_i^{0}\|^2
+ \frac{\sigma}{2 \gamma^L} \sum_{l=1}^{L} 
\xi^l  \|\m{u}_i^l - \m{u}_i^{l-1}\|^2 \nonumber \\
&&+ \frac{\rho}{2} \| \m{u} - \m{a}_i^L \|_{\m{Q}_i}^2
+ \frac{\mu_i}{2} \|\m{u} - \m{a}_i^L\|^2.
\nonumber
\end{eqnarray}
Since $\bar{L}_i^k(\m{u}) -  \bar{L}_i^k (\m{a}_i^{L}) \le 0$ when
$\m{u} = \bar{\m{x}}_i^k$ and since $\m{u}_i^0 = \m{x}_i^k$,
the proof is complete.
\bigskip

{\bf Proof of Lemma~\ref{L-key-lemma3}.}
Let us insert in (\ref{generalu})
$L = l_i^k$, the terminating value for $l$ in Algorithm~\ref{3}.
In addition,
substituting $\m{u} = \m{x}_i^*$, $\xi^l = 1$, $\m{a}_i^L = \m{z}_k^k$, and
$\m{u}_i^L = \m{x}_i^{k+1}$, we obtain
\begin{equation}\label{jumppoint}
L_i^k(\m{x}_i^*) - F_i^k(\m{z}_i^k) \ge
\frac{1}{2\Gamma_i^k} \left(
\|\m{x}_{e,i}^{k+1}\|^2 - \|\m{x}_{e,i}^k\|^2 \right)
+ \frac{\sigma}{2 \Gamma_i^k} \sum_{l=1}^{l_i^k} 
\|\m{u}_i^l - \m{u}_i^{l-1}\|^2 ,
\end{equation}
where $\m{x}_e^k = \m{x}_i^k - \m{x}_i^*$ and
\[
F_i^k(\m{z}_i^k) =
L_i^k(\m{z}_i^k) + \frac{\rho}{2} \|\m{A}_i( \m{z}_i^k - \m{x}_i^*) \|^2
+ \rho (\m{z}_i^k - \m{y}_i^k) \tr \bar{\m{Q}}_i (\m{z}_i^k - \m{x}_i^*)
+ \frac{\mu_i}{2}\|\m{z}_i^k - \m{x}_i^*\|^2.
\]
Now, by the definition of $L_i^k$, a Taylor expansion yields
\begin{eqnarray}
& L_i^k (\m{x}_i^*)  - L_i^k (\m{z}_i^{k}) = &
\label{zzz9} \\
& f_i(\m{x}_i^*) + h_i(\m{x}_i^*) - f_i(\m{z}_i^{k}) - h_i(\m{z}_i^{k})
- \rho\langle \m{A}_i\m{x}_i^{*}
- \m{b}_i^k + \g{\lambda}^k/\rho, \m{A}_i\m{z}_{e,i}^{k}\rangle 
- \frac{\rho}{2} \|\m{A}_i\m{z}_{e,i}^{k}\|^2,
\nonumber
\end{eqnarray}
where $\m{z}_{e}^k = \m{z}^k - \m{x}^*$.
Observe that
\begin{eqnarray*}
\m{A}_i \m{x}_i^* - \m{b}_i^k &=&
-\m{A}_i \m{z}_{e,i}^k + \m{A}_i\m{z}_i^k - \m{b}
+ \sum_{j < i} \m{A}_j \m{z}_{j}^{k} + \sum_{j > i} \m{A}_j \m{y}_{j}^k  \\
&=& 
-\m{A}_i \m{z}_{e,i}^k
+ \sum_{j \le i} \m{A}_j \m{z}_{e,j}^{k} + \sum_{j > i} \m{A}_j \m{y}_{e,j}^k .
\end{eqnarray*}
where $\m{y}_{e}^k = \m{y}^k - \m{x}^*$.
With this substitution in (\ref{zzz9}), we deduce that
\begin{eqnarray}
& L_i^k (\m{x}_i^*)  - L_i^k (\m{z}_i^{k})
- \frac{\rho}{2} \|\m{A}_i\m{z}_{e,i}^{k}\|^2 = &
\label{zzz2} \\
& - \Delta_i^k - \rho \left\langle \sum_{j \le i} \m{A}_j \m{z}_{e,j}^{k} +
\sum_{j > i} \m{A}_j \m{y}_{e,j}^k 
+ \g{\lambda}_e^k/\rho, \; \m{A}_i \m{z}_{e,i}^{k} \right\rangle,&
\nonumber
\end{eqnarray}
where
$\g{\lambda}_e^k = \g{\lambda}^k - \g{\lambda}^*$, and
\begin{equation} \label{Delta-ik}
\Delta_i^k =
f_i(\m{z}_i^{k}) + h_i(\m{z}_i^{k}) -  f_i(\m{x}_i^*) - h_i(\m{x}_i^*)
+ (\g{\lambda}^*, \m{A}_i\m{z}_{e,i}^{k}).
\end{equation}
Hence, we have
\begin{eqnarray*}
& L_i^k (\m{x}_i^*)- F_i^k (\m{z}_i^k) = L_i^k (\m{x}_i^*) - L_i^k (\m{z}_i^k)
- \frac{\rho}{2} \| \m{A}_i \m{z}_{e,i}^{k}\|^2 - 
\rho (\m{z}_i^k - \m{y}_i^k) \tr \bar{\m{Q}}_i  \m{z}_{e,i}^k
-\frac{\mu_i}{2} \|\m{z}_{e,i}^k\|^2& \\
&= - \Delta_i^k
- \rho \left\langle \sum_{j \le i} \m{A}_j \m{z}_{e,j}^{k} +
\sum_{j > i} \m{A}_j \m{y}_{e,j}^k 
+ \g{\lambda}_e^k/\rho, \; \m{A}_i \m{z}_{e,i}^{k} \right\rangle 
- \rho (\m{z}_i^k - \m{y}_i^k) \tr \bar{\m{Q}}_i  \m{z}_{e,i}^k
- \frac{\mu_i}{2} \|\m{z}_{e,i}^k\|^2&
\end{eqnarray*}
where $\Delta_i^k$ is defined in (\ref{Delta-ik}).
Combining this with the lower bound (\ref{jumppoint}) gives
\begin{eqnarray}
&- \rho \left\langle \sum_{j \le i} \m{A}_j \m{z}_{e,j}^{k} +
\sum_{j > i} \m{A}_j \m{y}_{e,j}^k
+ \g{\lambda}_e^k/\rho, \; \m{A}_i \m{z}_{e,i}^{k} \right\rangle - \Delta_i^k \ge &
\label{combine} \\
&\frac{1}{2\Gamma_i^k} (\|\m{x}_{e,i}^{k+1}\|^2 - \|\m{x}_{e,i}^k\|^2)
+ \frac{\sigma}{2\Gamma_i^k}
\sum_{l = 1}^{l_i^k} \|\m{u}_{ik}^{l} - \m{u}_{ik}^{l-1}\|^2& \nonumber \\
&+ \rho (\m{z}_i^k - \m{y}_i^k) \tr \bar{\m{Q}}_i  \m{z}_{e,i}^k
+ \frac{\mu_i}{2} \|\m{z}_{e,i}^k\|^2.&
\nonumber
\end{eqnarray}

Focusing on the left side of (\ref{combine}), observe that
\begin{eqnarray}
\sum_{j \le i} \m{A}_j \m{z}_{e,j}^{k} +
\sum_{j >i} \m{A}_j \m{y}_{e,j}^k &=&
\sum_{j =1 }^m \m{A}_j (\m{z}_{j}^{k} - \m{x}_j^*) + 
\sum_{j >i} \m{A}_j (\m{y}_{j}^k - \m{z}_j^{k}) \nonumber \\
&=& \m{A}\m{z}^{k}-\m{b} + 
\sum_{j >i} \m{A}_j (\m{y}_{j}^k - \m{z}_j^{k}) \label{xyz}
\end{eqnarray}
since $\m{Ax}^* = \m{b}$.
Let $\tau_i^k$ denote the first part of the right side of (\ref{combine});
that is
\[
\tau_i^k = \frac{1}{2\Gamma_i^k} (\|\m{x}_{e,i}^{k+1}\|^2 - \|\m{x}_{e,i}^k\|^2)
+ \frac{\sigma}{2\Gamma_i^k}
\sum_{l = 1}^{l_i^k} \|\m{u}_{ik}^{l} - \m{u}_{ik}^{l-1}\|^2
\]
With this notation and with the simplification (\ref{xyz}),
(\ref{combine}) becomes
\begin{eqnarray}\label{xzz}
 && -\rho \left\langle \m{A}_i \m{z}_{e,i}^{k},
\m{A}\m{z}^{k}-\m{b} + \g{\lambda}_e^k/\rho +
\sum_{j >i} \m{A}_j (\m{y}_{j}^k - \m{z}_j^{k}) \right\rangle \\
&\ge& \tau_i^k + \Delta_i^k +
\rho (\m{z}_i^k - \m{y}_i^k) \tr \bar{\m{Q}}_i \m{z}_{e,i}^k \nonumber \\
&=& \tau_i^k + \Delta_i^k +
\rho (\m{z}_i^k - \m{y}_i^k)\tr (\m{Q}_i - \rho \m{A}_i \tr \m{A}_i)
\m{z}_{e,i}^k
+ \frac{\mu_i}{2} \|\m{z}_{e,i}^k\|^2,
 \nonumber 
\end{eqnarray}
which gives
\begin{eqnarray}\label{xzz-1}
 && -\rho \left\langle \m{A}_i \m{z}_{e,i}^{k},
\m{A}\m{z}^{k}-\m{b} + \g{\lambda}_e^k/\rho +
\sum_{j \ge i} \m{A}_j (\m{y}_{j}^k - \m{z}_j^{k}) \right\rangle \\
&\ge& \tau_i^k + \Delta_i^k +
\rho (\m{z}_i^k - \m{y}_i^k) \tr \m{Q}_i \m{z}_{e,i}^k
+ \frac{\mu_i}{2} \|\m{z}_{e,i}^k\|^2.
\nonumber 
\end{eqnarray}

We will sum the inequality (\ref{xzz-1}) over $i$ between 1 and $m$.
Let $\m{r}^k = \m{Az}^{k} - \m{b}$, the residual for the linear system.
As in (\ref{xyz}), it follows that
\begin{equation}\label{1st}
\sum_{i = 1}^m \left\langle \m{A}_i \m{z}_{e,i}^{k}, \m{r}^k
+ \g{\lambda}_e^k/\rho \right\rangle =
\left\langle \m{r}^k, \m{r}^k + \g{\lambda}_e^k/\rho \right\rangle.
\end{equation}
Also, observe that
\[
\sum_{j \ge i} \m{A}_j (\m{y}_{j}^k - \m{z}_j^{k}) =
\sum_{j=1}^m \m{A}_j (\m{y}_{j}^k - \m{z}_j^{k}) -
\sum_{j =1}^{i-1} \m{A}_j (\m{y}_{j}^k - \m{z}_j^{k}),
\]
with the convention that the sum from $j =1$ to $j = 0$ is 0.
%
Hence, we have
\begin{eqnarray}  \label{2nd}
& &\sum_{i=1}^m \left( \left\langle \m{A}_i \m{z}_{e,i}^{k},
\sum_{j \ge i} \m{A}_j (\m{y}_{j}^k - \m{z}_j^{k}) \right\rangle
- (\m{y}_i^k - \m{z}_i^k) \tr \m{Q}_i \m{z}_{e,i}^k \right)  \\
&= & \left\langle \m{r}^k, 
\sum_{j=1}^m \m{A}_j \m{w}_{j} \right\rangle
- \sum_{i=1}^m  \left\langle  \m{z}_{e,i}^{k},
 \sum_{j=1}^{i-1} \m{A}_i \tr \m{A}_j \m{w}_{j} +
\m{Q}_i\m{w}_i  \right\rangle  \nonumber \\
& = &\left\langle \m{r}^k, 
\sum_{j=1}^m \m{A}_j \m{w}_{j} \right\rangle
- (\m{z}_e^{k})\tr \m{M} \m{w}, \nonumber
\end{eqnarray}
where $\m{M}$ is defined in (\ref{m-def}) and 
$\m{w} = \m{y}^k -\m{z}^{k}$.
We sum (\ref{xzz-1}) over $i$ between $1$ and $m$ and utilize
 (\ref{1st}) and (\ref{2nd}) to obtain
\begin{eqnarray}
&(\m{y}_e^k)\tr\m{Mw}
- \frac{1}{\rho} \left( \langle \m{r}^k , \g{\lambda}_e^k \rangle
+ \sum_{i=1}^m (\tau_i^k + \Delta_i^k)\right)&
\label{close} \\
&\ge \m{w}\tr\m{Mw} +
\left\langle
\m{r}^{k} ,
\m{r}^{k} + \sum_{j=1}^m \m{A}_j \m{w}_j
\right\rangle
+ \sum_{i=1}^m \frac{\mu_i}{2} \|\m{z}_{e,i}^k\|^2.  & \nonumber
\end{eqnarray}
Observe that $\m{M} + \m{M}\tr - \m{Q} \succeq \m{A} \tr \m{A}$, since
$\m{Q}_i \succeq \m{A}_i \tr \m{A}_i$ for all $1 \le i \le m$.  
Consequently, we have
\begin{eqnarray*}
\m{w}\tr\m{Mw} &=& \frac{1}{2} \m{w}\tr(\m{M} + \m{M}\tr)\m{w} =
\frac{1}{2} \m{w}\tr(\m{M} + \m{M}\tr - \m{Q})\m{w} +
\frac{1}{2}\m{w}\tr\m{Qw} \\
&\ge & \frac{1}{2}\left\| \sum_{i=1}^m \m{A}_i \m{w}_i\right\|^2
+ \frac{1}{2} \m{w}\tr\m{Qw},
\end{eqnarray*}
which implies that
\begin{eqnarray*}
&\m{w}\tr\m{Mw} +
\left\langle
\m{r}^{k} ,
\m{r}^{k} + \sum_{j=1}^m \m{A}_j \m{w}_j
\right\rangle \ge &
\frac{1}{2} \left(
\m{w}\tr\m{Qw} + \|\m{r}^k\|^2 +
\left\| \m{r}^k + \sum_{i=1}^m \m{A}_i \m{w}_i\right\|^2 \right).
\end{eqnarray*}
Hence, it follows from (\ref{close}) that
\begin{eqnarray}
&(\m{y}_e^k)\tr\m{Mw}
- \frac{1}{\rho} \left( \langle \m{r}^k , \g{\lambda}_e^k \rangle
+ \sum_{i=1}^m (\tau_i^k + \Delta_i^k) \right)
\label{closer}& \\
&\ge
\frac{1}{2} \left( \|\m{w}\|_{\m{Q}}^2 + \|\m{r}^k\|^2 \right)
+ \sum_{i=1}^m \frac{\mu_i}{2} \|\m{z}_{e,i}^k\|^2.& \nonumber
\end{eqnarray}

Let $\m{P} = \m{MQ}^{-1} \m{M} \tr$ and recall that $\m{w} =$
$\m{y}^{k} - \m{z}^{k}$.
By the definition of $\m{y}^{k+1}$ and $\g{\lambda}^{k+1}$ in Step~3
of Algorithm~\ref{ADMMcommon}, we have
\begin{eqnarray*}
& \| \m{y}_e^k \|_{\m{P}}^2 -
\| \m{y}_e^{k+1}\|_{\m{P}}^2 +
\frac{1}{\rho^2} \| \g{\lambda}_e^k\|^2
- \frac{1}{\rho^2} \| \g{\lambda}_e^{k+1}\|^2 = & \\
& \| \m{y}_e^k \|_{\m{P}}^2 -
\| \m{y}_e^k - \alpha  \m{M}^{-\sf T} \m{Q}
\m{w} \|_{\m{P}}^2  +
\frac{1}{\rho^2} (\| \g{\lambda}_e^k\|^2
-   \| \g{\lambda}_e^k + \alpha \rho\m{r}^{k}\|^2) = & \\
& 2 \alpha (\m{y}_e^k)\tr\m{Mw}
- \alpha^2  \|\m{w}\|_{\m{Q}}^2 -
\frac{2 \alpha}{\rho} \langle \m{r}^k , \g{\lambda}_e^k \rangle
- \alpha^2 \|\m{r}^k\|^2. &
\end{eqnarray*}
On the right side of this equality, we utilize (\ref{closer})
multiplied by $2\alpha$ to conclude that
\begin{eqnarray}
& \| \m{y}_e^k \|_{\m{P}}^2 -
\| \m{y}_e^{k+1}\|_{\m{P}}^2 +
\frac{1}{\rho^2} (\| \g{\lambda}_e^k\|^2
- \| \g{\lambda}_e^{k+1}\|^2) - \frac{2\alpha}{\rho}
\sum_{i=1}^m (\tau_i^k + \Delta_i^k) \ge&
\label{drfc} \\
& \alpha(1-\alpha) (\|\m{y}^k - \m{z}^k\|_{\m{Q}}^2 + \|\m{r}^k\|^2)
+ \alpha \sum_{i=1}^m \mu_i \|\m{z}_{e,i}^k\|^2.& \nonumber
\end{eqnarray}
So, by the definition of $\tau_i^k$,
the identity $\Delta^k = \sum_{i=1}^m \Delta_i^k$,
the inequality (\ref{drfc}), and the relation
$\Gamma_i^{k+1} \ge \Gamma_i^{k}$ in Steps~1b and 1c,
it follows that (\ref{Ek-decay}) holds.
\bigskip

{\bf Proof of Theorem \ref{L-glob-thm3}.}
Since $\xi^l = \delta^l \alpha^l \gamma^l = 1$
when the parameters $\delta^l$ and $\alpha^l$ are chosen
according to either (\ref{AG_constant}) or (\ref{AG_linesearch}),
Lemma~\ref{L-key-lemma3} can be utilized.
For any $p > 0$, we sum the decay property of Lemma~\ref{L-key-lemma3} 
to obtain
\begin{equation}\label{Ej2}
\quad \quad \,
E_j \ge E_{j+p} + c \sum_{k=j}^{j+p-1} \left(
\|\m{y}^k - \m{z}^{k}\|_{\m{Q}}^2
+ \|\m{Az}^{k} - \m{b}\|^2 + R^k \right), 
\end{equation}
where $c = \alpha \min\{\sigma, \rho (1-\alpha)\} > 0$.
Let $p$ tend to $+\infty$.
Since $\m{Q}$ is positive definite and $R^k \ge 0$, 
it follows from (\ref{Ej2}) that
\begin{equation}\label{L-lim12}
\lim_{k \to \infty} \|\m{y}^{k}- \m{z}^{k}\|= 0 =
\lim_{k \to \infty} \|\m{Az}^{k}- \m{b}\|.
\end{equation}
Moreover, by the definition of $E_k$ and Lemma~\ref{L-key-lemma3}, we know
$\m{y}^k$ and $\g{\lambda}^k$ are bounded sequences, and by the
first equation in (\ref{L-lim12}), $\m{z}^k$ is also a bounded sequence.
Hence, there exist an infinite sequence $\C{K} \subset \{1, 2, \ldots \}$
and limits $\m{x}^*$ and $\g{\lambda}^*$ such that
\begin{equation}\label{limz}
\lim_{k\in\C{K}} \m{z}^k = \m{x}^* \quad \mbox{and} \quad
\lim_{k\in\C{K}} \g{\lambda}^k = \g{\lambda}^*.
\end{equation}
By the first equation in (\ref{L-lim12}), we have
\begin{equation}\label{limy}
\lim_{k \in \C{K}} \m{y}^k = \m{x}^*.
\end{equation}
By the second equation in (\ref{L-lim12}), $\m{Ax}^* = \m{b}$.
Consequently, by (\ref{limz}) and (\ref{limy}),
\begin{equation}\label{L-lim-b2}
\lim_{k \in \C{K}} \left( \m{A}_i \m{z}_i^{k} -\m{b}_i^k \right) =
\lim_{k \in \C{K}} \left(
\sum_{j \le i }  \m{A}_j \m{z}_j^{k} +
\sum_{j > i } \m{A}_j \m{y}_j^k - \m{b} \right) = \m{Ax}^* - \m{b} = \m{0}
\end{equation}
for all $i \in [1,m]$.

The decay property (\ref{Ej2}) also implies that for each $i$,
\begin{equation}\label{riklim}
\lim_{k \to \infty} r_i^k =
\lim_{k \to \infty} \frac{1}{\Gamma_i^k} \sum_{l=1}^{l_i^k}
\|\m{u}_{ik}^l- \m{u}_{ik}^{l-1}\|^2 = 0,
\end{equation}
where $l_i^k$ is the terminating value of $l$ in Step~1b, which
exists since the parameters $\delta^l$ and $\alpha^l$
in Algorithm~$\ref{3}$ are chosen
according to either $(\ref{AG_constant})$ or $(\ref{AG_linesearch})$.
Combine this with (\ref{L-lim12}) to conclude that the parameter
$\epsilon^k$ in Step~2 of Algorithm~\ref{ADMMcommon} satisfies
\begin{equation}\label{eklim}
\lim_{k \to \infty} \epsilon^k = \lim_{k \to \infty} \psi(\epsilon^k) = 0.
\end{equation}
%
The remainder of the proof is partitioned into two cases depending on whether
the monotone nondecreasing sequence $\Gamma_i^k$
either approaches a finite limit, or tends to infinity.

{\bf Case 1.} For some $i$, $\Gamma_i^k$ approaches a finite limit.
In \cite[pp.~227--228]{HagerZhang19} it is shown that
$\gamma^l \ge l^2 \Theta$ for some constant $\Theta > 0$, independent of $k$.
Since $\Gamma_i^k = \gamma^l$ for some $l$, it follows that
$l_i^k$, the terminating value $l$ in Step~1 of Algorithm~\ref{3},
is uniformly bounded when $\Gamma_i^k$ approaches a finite limit.
By (\ref{riklim}), $\| \m{u}_{ik}^l - \m{u}_{ik}^{l-1} \|$ approaches zero,
where the convergence is uniform in $k$ and $l \in [1, l_i^k]$.
Since $\m{u}_{ik}^0 = \m{x}_i^k$, the triangle inequality and the
uniform upper bound for $l_i^k$ imply that
$\| \m{x}_i^k - \m{u}_{ik}^l \|$ approaches zero,
where the convergence is uniform in $k$ and $l \in [1, l_i^k]$.
Since $\m{a}_{ik}^l$ is a convex combination of
$\m{u}_{ik}^l$ (see (\ref{convex2})) for $0 \le l \le l_i^k$ with $l_i^k$
uniformly bounded and $\| \m{x}_i^k - \m{u}_{ik}^l \|$ approaching zero,
it follows that $\|\m{a}_{ik}^l - \m{x}_i^k\|$ approaches zero.
Since $\m{z}_i^k = \m{a}_{ik}^{l_i^k}$ and since $\m{z}_i^k$ approaches
$\m{x}_i^*$ as $k\in\C{K}$ tends to infinity, we deduce that
\begin{equation}\label{uiklim}
\m{x}_i^* = \lim_{k \in \C{K}} \m{y}_i^k
= \lim_{k \in \C{K}} \m{z}_i^k
= \lim_{k \in \C{K}} \m{x}_i^k
= \lim_{k \in \C{K}} \m{u}_{ik}^l
= \lim_{k \in \C{K}} \m{a}_{ik}^l
= \lim_{k \in \C{K}} \bar{\m{a}}_{ik}^l ,
\end{equation}
%
where the last equality is due to the fact that
$\bar{\m{a}}_{ik}^l$ is a convex combination of
$\m{a}_{ik}^{l-1}$ and $\m{u}_{ik}^{l-1}$.

In (\ref{h2}) we give the first-order optimality condition for
$\m{u}_{ik}^l$.
Taking the limit as $k \in \C{K}$ tends to infinity and utilizing
(\ref{L-lim-b2}) and (\ref{uiklim}), we obtain
\begin{equation}\label{h3}
h_i (\m{x}_i^*) \le h_i(\m{u}) + \langle \nabla f_i(\m{x}_i^*)
+ \m{A}_i\tr \g{\lambda}^*, \m{u}-\m{x}_i^* \rangle
\end{equation}
for every $\m{u} \in \mathbb{R}^{n_i}$.
Since $\m{Ax}^* = \m{b}$ and the first-order optimality conditions are
both necessary and sufficient for optimality in this convex setting,
it would follow that $(\m{x}^*, \g{\lambda}^*) \in \C{W}^*$
if (\ref{h3}) holds for every $i \in [1, m]$.
To show that (\ref{h3}) holds for all $i$, we need to consider
the situation where $\Gamma_i^k$ tends to infinity.

{\bf Case 2.} Suppose that $\Gamma_i^k$ approaches infinity.
Let $\bar{\m{x}}_i^k$ be the minimizer of $\bar{L}_i^k$ defined in
(\ref{barlik}).
Observe that minimizing
$\bar{L}_i^k (\m{u})$ over $\m{u} \in \mathbb{R}^{n_i}$
is equivalent to minimizing a sum of the form
$g(\m{u}) + h_i(\m{u}) + \langle \m{u}, \m{c}^k \rangle$ where
\[
\m{c}^k =
\rho \m{A}_i\tr(\m{A}_i \m{y}_i^k - \m{b}_i^k +
\g{\lambda}^k/\rho) - \rho \m{Q}_i \m{y}_i^k,
\]
and $g(\m{u}) = f_i(\m{u}) + 0.5 \rho \|\m{u} \|_{\m{Q}_i}^2$.
Note that $g$ is smooth and satisfies a strong convexity condition
\begin{equation}\label{strong}
(\m{u} - \m{v})\tr (\nabla g(\m{u}) - \nabla g(\m{v})) \ge
\rho \nu_i \|\m{u} - \m{v}\|^2,
\end{equation}
where $\nu_i > 0$ is the smallest eigenvalue of $\m{Q}_i$.
By the strong convexity of $\bar{L}_i^k$, it has a unique minimizer, and
from the first-order optimality conditions and the strong convexity
condition (\ref{strong}), we obtain the bound
\begin{equation}\label{lipbound}
\|\bar{\m{x}}_i^j - \bar{\m{x}}_i^k\| \le \|\m{c}^j - \m{c}^k\|/(\rho \nu_i).
\end{equation}
Since $\m{z}^k$, $\m{y}^k$, and $\g{\lambda}^k$ are bounded sequences,
it follows that $\bar{\m{x}}_i^k$ is a bounded sequence.
For $k \in \C{K}$, the sequences
$\m{z}^k$, $\m{y}^k$, and $\g{\lambda}^k$ converge to
$\m{x}^*$, $\m{x}^*$, and $\g{\lambda}^*$ respectively and $\m{Ax}^*= \m{b}$,
which implies that
\begin{eqnarray}\label{*}
\m{c}^* &= & \lim_{k\in\C{K}} \m{c}^k =
\m{A}_i\tr \left[ \rho \m{A}_i  \m{x}_i^* + \g{\lambda}^* - \rho \left( \m{b}
- \sum_{j\ne i} \m{A}_j \m{x}_j^* \right) \right]  - \rho \m{Q}_i  \m{x}_i^* \\
& = & \m{A}_i\tr \g{\lambda}^*  - \rho \m{Q}_i  \m{x}_i^*. \nonumber
\end{eqnarray}
%
Consequently, by (\ref{lipbound}),
$\bar{\m{x}}_i^k$ for $k \in \C{K}$ forms a Cauchy
sequence which approaches a limit.

By (\ref{eklim}) and the stopping condition in Algorithm~(\ref{3}),
$\|\m{x}_i^{k} - \m{z}_i^k\|/\sqrt{\Gamma_i^k}$
tends to zero as $k$ tends to infinity.
By (\ref{def-Ek}) and (\ref{Ej2}), $\g{\lambda}^k$ and $\m{y}^k$ are bounded,
which implies that $\m{z}^k$ is bounded by (\ref{L-lim12}).
By (\ref{lipbound}), $\bar{\m{x}}_i^k$ is also bounded.
Since $\Gamma_i^k$ tends to infinity in Case~2 and
$\|\m{x}_i^{k} - \m{z}_i^k\|/\sqrt{\Gamma_i^k}$ tends to zero,
it follows from the triangle inequality and the boundedness of
$\bar{\m{x}}_i^k$ and $\m{z}^k$ that
$\|\m{x}_i^{k} - \bar{\m{x}}_i^k\|/\sqrt{\Gamma_i^k}$ tends to zero as $k$
tends to infinity.
Hence, by Lemma~\ref{lem-AG-Conv}, $\m{z}_i^k = \m{a}_i^{l_i^k}$ approaches
$\bar{\m{x}}_i^k$ as $k$ tends to infinity;
since $\m{z}_i^k$ approaches $\m{x}_i^*$ as $k \in\C{K}$ tends to infinity,
it follows that $\bar{\m{x}}_i^k$ approaches $\m{x}_i^*$ as $k \in \C{K}$
tends to infinity.
Let $\bar{\m{x}}_i^*$ be defined by
\[
\bar{\m{x}}_i^* =
\arg \min_{\m{u}} \{ g(\m{u}) + h_i(\m{u}) +
\langle \m{u}, \m{c}^* \rangle \}.
\]
By (\ref{lipbound}) and the fact that
$\bar{\m{x}}_i^k$ approaches $\m{x}_i^*$ as
$k \in \C{K}$ tends to infinity, we conclude that
$\bar{\m{x}}_i^* = \m{x}_i^*$.
In summary, we have
\begin{eqnarray}
\lim_{k \in \C{K}} \bar{\m{x}}_i^k &=& \m{x}_i^* = \bar{\m{x}}_i^* =
\arg \min_{\m{u}} \{ g(\m{u}) + h_i(\m{u}) +
\langle \m{u}, \m{c}^* \rangle \}. \nonumber \\
&=& \arg \min_{\m{u}} \{ f_i(\m{u})
+ 0.5 \rho \| \m{u}\|_{\m{Q}_i}^2 + h_i (\m{u})
+ \langle \m{A}_i \tr \g{\lambda}^*
- \rho \m{Q}_i \m{x}_i^* , \m{u} \rangle \} .
\label{infinity}
\end{eqnarray}
The first-order optimality conditions for (\ref{infinity}) are
exactly the same as (\ref{h3}).
This shows that (\ref{h3}) holds in Case~1 and Case~2,
and $\m{x}^*$ is an optimal solution of (\ref{Prob})--(\ref{ProbM})
with associated multiplier $\g{\lambda}^*$.

Finally, we need to show that the entire sequence converges.
If $\Gamma_i^k$ is uniformly bounded as in Case~1, then by (\ref{uiklim}),
$\m{x}_i^k$ approaches $\m{x}_i^*$ and
$\|\m{x}_i^k - \m{x}_i^*\|^2/\Gamma_i^k$ approaches zero
as $k$ tends to infinity with $k \in \C{K}$.
On the other hand, when
$\Gamma_i^k$ tends to infinity as in Case~2, we showed that
$\|\m{x}_i^k - \bar{\m{x}}_i^k\|^2/\Gamma_i^k$ approaches zero and
$\bar{\m{x}}_i^k$ approaches $\m{x}_i^*$ when $k \in \C{K}$ tends to infinity.
Hence, $\|\m{x}_i^k - \m{x}_i^*\|^2/\Gamma_i^k$ approaches zero when
$k \in \C{K}$ tends to infinity.
Thus in Case~1 and Case~2,
$\|\m{x}_i^k - \m{x}_i^*\|^2/\Gamma_i^k$ approaches zero as $k \in \C{K}$ tends
to infinity.
By the definition of $E_k$ in (\ref{Ek-decay}), $E_k$ tends to zero as
$k \in \C{K}$ tends to infinity.
Letting $j$ tend to infinity in (\ref{Ej2}) with $j \in \C{K}$,
it follows that $E_j$ approaches zero, while the right side
of (\ref{Ej2}) shows that the entire sequence $(\m{y}^k, \g{\lambda}^k)$
approaches $(\m{x}^*, \g{\lambda}^*)$.
By (\ref{L-lim12}), the $\m{z}^k$ sequence also approaches $\m{x}^*$.
This completes the proof.
\newpage

\begin{thebibliography}{10}

\bibitem{Bertsekas2003}
{\sc D.~P. Bertsekas}, {\em Convex Analysis and Optimization}, Athena
  Scientific, Belmont, MA, 2003.

\bibitem{Boyd10}
{\sc S.~Boyd, N.~Parikh, E.~Chu, B.~Peleato, and J.~Eckstein}, {\em Distributed
  optimization and statistical learning via the alternating direction method of
  multipliers}, Machine Learning, 3 (2010), pp.~1--122.

\bibitem{CaiHanYuan17}
{\sc X.~Cai, D.~Han, and X.~Yuan}, {\em On the convergence of the direct
  extension of {ADMM} for three-block separable convex minimization models with
  one strongly convex function}, Comput. Optim. Appl., 66 (2017), pp.~39--73.

\bibitem{chyy2016}
{\sc C.~Chen, B.~He, Y.~Ye, and X.~Yuan}, {\em The direct extension of {ADMM}
  for multi-block convex minimization problems is not necessarily convergent},
  Math. Program., 155 (2016), pp.~57--79.

\bibitem{ChenLiLiuYe15}
{\sc C.~Chen, M.~Li, X.~Liu, and Y.~Ye}, {\em On the convergence of multi-block
  alternating direction method of multipliers and block coordinate descent
  method},  (2015, arXiv: 1508.00193).

\bibitem{ChenShenYou13}
{\sc C.~Chen, Y.~Shen, and Y.~You}, {\em On the convergence analysis of the
  alternating direction method of multipliers with three blocks}, Abstr. Appl.
  Anal., 2013 (2013).

\bibitem{ChenTeboulle94}
{\sc G.~Chen and M.~Teboulle}, {\em A proximal-based decomposition method for
  convex minimization problems}, Math. Programming, 64 (1994), pp.~81--101.

\bibitem{CWHLv19}
{\sc J.~W. Chen, Y.~Y. Wang, H.~J. He, and Y.~B. Lv}, {\em Convergence analysis
  of positive-indefinite proximal {ADMM} with a {Glowinski's} relaxation
  factor}, Numer. Algor.,  (2019, DOI: 10.1007/s11075-019-00731-9).

\bibitem{ChenSunToh2015}
{\sc L.~Chen, D.~Sun, and K.~Toh}, {\em An efficient inexact symmetric
  {Gauss-Seidel} based majorized {ADMM} for high-dimensional convex composite
  conic programming}, Math. Program., 161 (2017), pp.~237--270.

\bibitem{chy13}
{\sc Y.~Chen, W.~W. Hager, M.~Yashtini, X.~Ye, and H.~Zhang}, {\em Bregman
  operator splitting with variable stepsize for total variation image
  reconstruction}, Comput. Optim. Appl., 54 (2013), pp.~317--342.

\bibitem{DavisYin15}
{\sc D.~Davis and W.~Yin}, {\em A three-operator splitting scheme and its
  optimization applications}, Set-Valued and Variational Analysis, 25 (2017),
  pp.~829--858.

\bibitem{EB92}
{\sc J.~Eckstein and D.~Bertsekas}, {\em On the {Douglas-Rachford} splitting
  method and the proximal point algorithm for maximal monotone operators},
  Mathematical Programming, 55 (1992), pp.~293--318.

\bibitem{Eckstein13}
{\sc J.~Eckstein and P.~J.~S. Silva}, {\em A practical relative error criterion
  for augmented {Lagrangians}}, Math. Program., 141 (2013), pp.~319--348.

\bibitem{EcksteinYao17}
{\sc J.~Eckstein and W.~Yao}, {\em Approximate {ADMM} algorithms derived from
  {Lagrangian} splitting}, Comput. Optim. Appl., 68 (2017), pp.~363--405.

\bibitem{EcksteinYao18}
\leavevmode\vrule height 2pt depth -1.6pt width 23pt, {\em Relative-error
  approximate versions of {Douglas–Rachford} splitting and special cases of
  the {ADMM}}, Math. Program., 170 (2018), pp.~417--444.

\bibitem{GM76}
{\sc D.~Gabay and B.~Mercier}, {\em A dual algorithm for the solution of
  nonlinear variational problems via finite-element approximations},
  Comput.~Math.~Appl., 2 (1976), pp.~17--40.

\bibitem{gl84}
{\sc R.~Glowinski}, {\em Numerical Methods for Nonlinear Variational Problems},
  Springer-Verlag, New York, 1984.

\bibitem{GoldfarbMa2012}
{\sc D.~Goldfarb and S.~Ma}, {\em Fast multiple-splitting algorithms for convex
  optimization}, {SIAM} J. Optim., 22 (2012), pp.~533--556.

\bibitem{GolTre1979}
{\sc E.~Gol’shtein and N.~Tret’yakov}, {\em Modified {Lagrangians} in
  convex programming and their generalizations}, in Point-to-Set Maps and
  Mathematical Programming, P.~Huard, ed., vol.~10 of Mathematical Programming
  Studies, Springer Berlin Heidelberg, 1979, pp.~86--97.

\bibitem{hyz16}
{\sc W.~W. Hager, M.~Yashtini, and H.~Zhang}, {\em An {O(1/k)} convergence rate
  for the variable stepsize {Bregman} operator splitting algorithm}, {SIAM} J.
  Numer. Anal., 54 (2016), pp.~1535--1556.

\bibitem{HagerZhang19}
{\sc W.~W. Hager and H.~Zhang}, {\em Inexact alternating direction methods of
  multipliers for separable convex optimization}, Comput. Optim. Appl., 73
  (2019), pp.~201--235.

\bibitem{HagerZhang19b}
\leavevmode\vrule height 2pt depth -1.6pt width 23pt, {\em Convergence rates
  for an inexact {ADMM} applied to separable convex optimization}, arXiv,
  (2020, DOI: 2001.02503).

\bibitem{HanSunZhang17}
{\sc D.~Han, D.~Sun, and L.~Zhang}, {\em Linear rate convergence of the
  {Alternating Direction Method of Multipliers} for convex composite
  programming}, Math. Oper. Res., 43 (2018), pp.~622--637.

\bibitem{HanYuan12}
{\sc D.~Han and X.~Yuan}, {\em A note on the alternating direction method of
  multipliers}, J. Optim. Theory Appl., 155 (2012), pp.~227--238.

\bibitem{HLHY2002}
{\sc B.~He, L.~Liao, D.~Han, and H.~Yan}, {\em A new inexact alternating
  directions method for monotone variational inequalities}, Math. Programming,
  92 (2002), pp.~103--118.

\bibitem{HTXY2013}
{\sc B.~He, M.~Tao, M.~Xu, and X.~Yuan}, {\em An alternating direction-based
  contraction method for linearly constrained separable convex programming
  problems}, Optimization, 62 (2013), pp.~573--596.

\bibitem{HeTaoXuYuan12}
{\sc B.~He, M.~Tao, and X.~Yuan}, {\em Alternating direction method with
  {Gaussian} back substitution for separable convex programming}, SIAM
  J.~Optim., 22 (2012), pp.~313--340.

\bibitem{HeYuan12}
{\sc B.~He and X.~Yuan}, {\em On the $\mathcal{O}(1/n)$ convergence rate of the
  {Douglas-Rachford} alternating direction method}, {SIAM} J. Numer. Anal., 50
  (2012), pp.~700--709.

\bibitem{HeMaYuan2020}
{\sc B.~S. He, F.~Ma, and X.~M. Yuan}, {\em Optimally linearizing the
  alternating direction method of multipliers for convex programming}, Comput.
  Optim. Appl., 75 (2020), pp.~361--388.

\bibitem{HongLuo2017}
{\sc M.~Hong and Z.~Luo}, {\em On the linear convergence of the alternating
  direction method of multipliers}, Math. Program., 162 (2017), pp.~165--199.

\bibitem{LiLiaoYuan2013}
{\sc M.~Li, L.~Liao, and X.~Yuan}, {\em Inexact alternating direction methods
  of multipliers with logarithmic-quadratic proximal regularization}, J. Optim.
  Theory Appl., 159 (2013), pp.~412--436.

\bibitem{LiSunToh14}
{\sc M.~Li, D.~Sun, and K.~C. Toh}, {\em A convergent 3-block semi-proximal
  {ADMM} for for convex minimization problems with one strongly convex block},
  Asia-Pacific Journal of Operational Research, 32 (2015), pp.~1--19.

\bibitem{LinMaZhang15}
{\sc T.~Lin, S.~Ma, and S.~Zhang}, {\em On the global linear convergence of the
  {ADMM} with multiblock variables}, SIAM J.~Optim., 25 (2015), pp.~1478--1497.

\bibitem{LiuYuanZengZhang2018}
{\sc Y.~Liu, X.~Yuan, S.~Zeng, and J.~Zhang}, {\em Partial error bound
  conditions for the linear convergence rate of the alternating direction
  method of multipliers}, {SIAM} J. Numer. Anal., 56 (2018), pp.~2095--2123.

\bibitem{RMS13}
{\sc R.~D.~C. Monteiro and B.~F. Svaiter}, {\em Iteration-complexity of
  block-decomposition algorithms and the alternating direction method of
  multipliers}, {SIAM} J. Optim., 23 (2013), pp.~475--507.

\bibitem{Robinson81}
{\sc S.~M. Robinson}, {\em Some continuity properties of polyhedral
  multifunctions}, Math. Prog. Study, 14 (1981), pp.~206--214.

\bibitem{rf76}
{\sc R.~T. Rockafellar}, {\em Monotone operators and the proximal point
  algorithm}, {SIAM} J. Control, 14 (1976), pp.~877--898.

\bibitem{ShefiTeboulle14}
{\sc R.~Shefi and M.~Teboulle}, {\em Rate of convergence analysis of
  decomposition methods based on the proximal method of multipliers for convex
  minimization}, {SIAM} J. Optim., 24 (2014), pp.~269--297.

\bibitem{SolSva2000}
{\sc M.~V. Solodov and B.~F. Svaiter}, {\em An inexact hybrid generalized
  proximal point algorithm and some new results on the theory of {Bregman}
  functions}, Math. Oper. Res., 25 (2000), pp.~214--230.

\bibitem{TaoYuan2011}
{\sc M.~Tao and X.~Yuan}, {\em Recovering low-rank and sparse components of
  matrices from incomplete and noisy observations}, {SIAM} J. Optim., 21
  (2011), pp.~57--81.

\bibitem{wgy2010}
{\sc Z.~Wen, D.~Goldfarb, and W.~Yin}, {\em Alternating direction augmented
  {Lagrangian} methods for semidefinite programming}, Math. Prog. Comput., 2
  (2010), pp.~203--230.

\bibitem{YangHan16}
{\sc W.~H. Yang and D.~Han}, {\em Linear convergence of the alternating
  direction method of multipliers for a class of convex optimization problems},
  {SIAM} J. Numer. Anal., 54 (2016), pp.~625--640.

\bibitem{YuanZengZhang2020}
{\sc X.~Yuan, S.~Z. Zeng, and J.~Zhang}, {\em Discerning the linear convergence
  of admm for structured convex optimization through the lens of variational
  analysis}, J. Mach. Learn. Res., to appear (2020).

\end{thebibliography}

\end{document}